\documentclass{article}

\usepackage[margin=1in]{geometry}

\usepackage{amsmath}
\usepackage{amsfonts}
\usepackage{amssymb}    
\usepackage{amsthm}

\usepackage{enumitem}   
\usepackage{booktabs}   
\usepackage{multirow}

\usepackage{graphicx}  
\usepackage{algorithm}  
\usepackage{algorithmic}
\usepackage[numbers,sort&compress]{natbib}

\usepackage{comment}
\usepackage{authblk}
\usepackage{xcolor}   
\usepackage{colortbl}
\usepackage{booktabs} 
\usepackage{makecell}
\usepackage{hyperref}
\hypersetup{
  breaklinks,
  colorlinks,
  citecolor={green!50!black},
  urlcolor={red!50!black},
  linkcolor={green!50!black}
}
\usepackage[capitalize,nameinlink]{cleveref} 
\usepackage{amsthm}
\usepackage{threeparttable}

\crefname{assumption}{Assumption}{Assumptions}
\Crefname{assumption}{Assumption}{Assumptions}
\newtheorem{assumption}{Assumption}
\newtheorem{remark}{Remark}
\newtheorem{theorem}{Theorem}
\newtheorem{lemma}{Lemma}
\newtheorem{proposition}{Proposition}
\newtheorem{definition}{Definition}

\title{\textbf{Adaptive Barzilai-Borwein Proximal Gradient Method for Nonconvex Optimization }}

\author[1]{Yinuo Li}
\author[1]{Na Huang\thanks{Research partially supported by National Natural Science Foundation of China (No.\,12001531).}}
\author[1]{Ruizhi Zhou}

\affil[1]{Department of Applied Mathematics, College of Science, China Agricultural University, Beijing, China.}

\affil{E-mail: \texttt{lnuo@cau.edu.cn, hna@cau.edu.cn., rzzhou@cau.edu.cn
}}

\begin{document}
\maketitle
\begin{abstract}

 The Barzilai–Borwein (BB) method is an efficient gradient-based approach for unconstrained optimization that approximates spectral information of the Hessian matrix to capture curvature at low computational cost. In this paper, we extend the BB stepsize strategy to composite nonconvex optimization problems consisting of a smooth nonconvex term and a proper closed convex term, and propose an adaptive Barzilai–Borwein proximal gradient method for nonconvex optimization (AdaBBNC). The proposed method incorporates a flexible BB-based curvature estimate into the proximal gradient framework to enhance adaptability in nonconvex settings. Under mild assumptions, we establish that AdaBBNC achieves the optimal iteration complexity of $\mathcal{O}(\epsilon^{-2})$ for finding an $\epsilon$-stationary point, without requiring any prior knowledge of the global Lipschitz constant. Numerical experiments demonstrate the effectiveness and robustness of the proposed method. Compared with recent parameter-free and linesearch-free adaptive proximal gradient methods, AdaBBNC exhibits more aggressive yet stable behavior in ill-conditioned optimization problems.


\end{abstract}

\vspace{1em}

\vspace{1em}

\section{Introduction}

Composite optimization problems have attracted considerable attention due to their broad applications in machine learning and high-dimensional statistics \cite{candesExactMatrixCompletion2012,candes2008introduction,hastieStatisticalLearningSparsity2015,tibshiraniRegressionShrinkageSelection1996,wrightSparseReconstructionSeparable2009}. In this work, we consider the following class of composite optimization problems:
\begin{equation}\label{eq:composite_problem}
    \min_{x \in \mathbb{R}^n} F(x) := f(x) + g(x),
\end{equation}
where $f: \mathbb{R}^n \rightarrow \mathbb{R}$ is a continuously differentiable, but possibly nonconvex function, and $g: \mathbb{R}^n \rightarrow \mathbb{R} \cup \{+\infty\}$ is a proper, lower semi-continuous, and convex function. 

Over the past decades, a broad range of methods has been developed for solving problem \eqref{eq:composite_problem}. Among them, second-order methods, such as proximal Newton and quasi-Newton methods \cite{leeProximalNewtontypeMethods2014,stellaForwardBackwardQuasiNewton2017}, often exhibit rapid local convergence by exploiting curvature information through Hessian matrices or their approximations. To effectively decouple the smooth and nonsmooth components, operator splitting techniques have also been extensively investigated. Representative examples include the alternating direction method of multipliers and various proximal splitting schemes \cite{combettesProximalSplittingMethods2011,lionsSplittingAlgorithmsSum1979,neal2011distributed}, which provide effective frameworks for handling complex coupled nonsmooth regularizers. Additionally, smoothing techniques \cite{beckSmoothingFirstOrder2012,nesterovSmoothMinimizationNonsmooth2005} offer an alternative strategy by constructing differentiable approximations of the nonsmooth terms. However, as the dimensionality of modern datasets and the scale of optimization variables continue to grow, second-order methods often become computationally prohibitive due to the high cost associated with forming, storing, and inverting Hessian matrices or their approximations.

Consequently, the focus in large-scale optimization has increasingly shifted toward first-order methods, which rely solely on gradient evaluations and scale favorably with the problem dimension \cite{parikhProximalAlgorithms2014}. The proximal gradient method (PGM) \cite{beckFirstOrderMethodsOptimization2017,lionsSplittingAlgorithmsSum1979,passtyErgodicConvergenceZero1979} is one of the most classical and widely studied approaches. By fully leveraging the composite structure of $F(x)$, PGM generates iterates according to
\begin{equation}\label{eq:pgm_update}
    x_{k+1} = \text{prox}_{\alpha_k g}(x_k - \alpha_k \nabla f(x_k)),
\end{equation}
where $\alpha_k > 0$ denotes the stepsize at the $k$-th iteration, and ${\rm prox}_{\alpha_k g}(\cdot)$ is the proximal operator associated with $g$, defined by
\begin{equation*}
    \text{prox}_{\alpha_k g}(x) := \arg\min_{y \in \mathbb{R}^n} \left\{ g(y) + \frac{1}{2\alpha_k} \|y - x\|^2 \right\}.
\end{equation*}

Although the classical PGM admits a complete convergence analysis, its practical performance is often highly sensitive to stepsize selection. The standard PGM typically relies on the global Lipschitz constant of the gradient, usually denoted by $L$, to determine the stepsize \cite{beckFirstOrderMethodsOptimization2017,nesterovIntroductoryLecturesConvex2004}. However, obtaining a tight and reliable estimate of $L$ is generally difficult. Even when $L$ is available, the resulting constant stepsize $1/L$ is dictated by the worst-case curvature of the objective landscape, which can be overly conservative and lead to slow convergence in relatively flat regions. To address this issue, linesearch (or backtracking) strategies have been widely adopted as the \textit{de facto} standard \cite{armijoMinimizationFunctionsHaving1966,bellocruzConvergenceForwardBackward2016,bertsekas1997nonlinear,goldsteinCauchysMethodMinimization1962}. But linesearch procedures typically require multiple evaluations of the objective function at each iteration. For large-scale problems or models with expensive function evaluations, this additional computational cost can become a significant bottleneck.



Recent research has increasingly focused on adaptive stepsize strategies to mitigate the drawbacks of conservative constant stepsizes. The class of ``linesearch-free'' or ``adaptive'' techniques was first developed for variational inequalities \cite{malitskyProjectedReflectedGradient2015, malitskyGoldenRatioAlgorithms2020} and later extended to smooth convex optimization \cite{malitskyAdaptiveGradientDescent2019a,zhouAdaBBAdaptiveBarzilaiBorwein2024}. Along these lines, adaptive variants of PGM that dynamically exploit local curvature information without the computational burden of linesearch procedures, have been extensively investigated \cite{latafatAdaptiveProximalAlgorithms2023,malitskyAdaptiveProximalGradient2024a}. These methods typically estimate the local smoothness modulus by
\begin{equation}\label{L_k}
    L_k = \frac{\left\| \nabla f(x^k) - \nabla f(x^{k-1}) \right\|}{\left\| x^k - x^{k-1} \right\|},
\end{equation}
which has further inspired the development of more advanced methods. Nevertheless, their theoretical guarantees and initial applications were mainly tailored for convex or well-structured settings. For general nonconvex optimization problems, a critical ongoing challenge is the development of adaptive stepsize strategies capable of effectively and robustly navigating highly irregular and unpredictable local geometries \cite{lanOptimalParameterfreeGradient2026, yagishitaSimpleLinesearchfreeFirstorder2025a}. \citet{lanProjectedGradientMethods2024a} proposed the auto-conditioned projected gradient (AC-PG) method for nonconvex constrained optimization, where the stepsize is chosen as the reciprocal of $\gamma_t$, with
\begin{equation*}
    \gamma_t = \max\{\mathcal{L}_0, \dots, \mathcal{L}_{t-1}\} \quad \text{and} \quad \mathcal{L}_t = \frac{2(f(x_t)-f(x_{t-1})-\langle \nabla f(x_{t-1}), x_t-x_{t-1}\rangle)}{\|x_t-x_{t-1}\|^2} \,.
\end{equation*}
The problem in \cite{lanProjectedGradientMethods2024a} can be viewed as a special case of \eqref{eq:composite_problem} in which the nonsmooth term $g(x)$ is the indicator function of a closed convex set $\mathcal{X}$. Recently, \citet{yagishitaSimpleLinesearchfreeFirstorder2025a} extended this auto-conditioned stepsize to the general composite setting \eqref{eq:composite_problem} and introduced the auto-conditioned proximal gradient method (AC-PGM), thereby substantially broadening the applicability of linesearch-free adaptive stepsizes. Note that $\gamma_t$ is monotonically nondecreasing, which implies that the the corresponding stepsize is monotonically nonincreasing. As a result, once the stepsize is reduced due to a locally high-curvature region, it cannot increase again even when the iterates subsequently enter flatter regions of the objective landscape. This may lead to unnecessarily small stepsizes, thereby slowing practical convergence.



Nonmonotone adaptive stepsizes have also been studied for nonconvex optimization \cite{liuNonmonotoneAcceleratedProximal2024,yeSimpleAdaptiveProximal2025a}. Their convergence analysis often relies on Lyapunov functions to move beyond the traditional framework requiring restrictive monotonic descent of the objective function. A notable example is the adaptive proximal gradient method for nonconvex (AdaPGNC)  proposed in \cite{yeSimpleAdaptiveProximal2025a}, which distinguishes itself from traditional approaches by exploiting both local upper and lower curvature information. Under the weak convexity assumption, the method dynamically estimates the local lower curvature as follows:
\begin{equation}\label{l_k}
    \ell_k = \frac{2\left(f(x_k) - f(x_{k-1}) - \langle \nabla f(x_k), x_k - x_{k-1} \rangle\right)}{\|x_k - x_{k-1}\|^2} \text{}.
\end{equation}

This parameter provides an effective characterization of the local geometry: $\ell_k \le 0$ indicates that $f(x)$ exhibits locally convex behavior, whereas $\ell_k > 0$ signals the presence of local nonconvexity. By combining this estimate with the upper curvature approximation $L_k$ in \eqref{L_k}, AdaPGNC updates its stepsize $\lambda_k$ according to
\begin{equation}\label{AdaPGNC}
    \lambda_k =
    \begin{cases}
        \min\left\{\sqrt{1+\rho_{k-1}}\lambda_{k-1}, \frac{1}{L_k}\right\}, & \text{if } \ell_k \le 0;\\[4pt]
        \min\left\{\sqrt{1+\rho_{k-1}}\lambda_{k-1}, \frac{1}{\sqrt{2}L_k}, \sqrt{\frac{\lambda_{k-1}}{2\ell_k}}\right\}, & \text{otherwise} \text{}.
    \end{cases}
\end{equation}
The update rule \eqref{AdaPGNC} remains relatively conservative in nonconvex regions (i.e., when $\ell_k > 0$) as the strict scaling factors and multiple minimum operators impose a severe truncation on the stepsize, restricting the mehtod's ability to take larger and potentially more effective exploratory steps.

Therefore, it is natural to seek richer curvature approximations that permit more aggressive adaptive stepsize updates. This motivates us to consider the Barzilai-Borwein (BB) method \cite{barzilaiTwopointStepSize1988a}, whose stepsizes have strong nonmonotone behavior but significantly improve the performance of gradient methods \cite{raydanBarzilaiBorweinGradient1997,burdakovStabilizedBarzilaiborweinMethod2019}. Based on an approximate secant condition, the two classical BB stepsizes are given by
\begin{align}
    \text{(Long BB)} \quad \alpha_k^{\mathrm{BBL}} &= \dfrac{\|x_k - x_{k-1}\|^2}{\langle \nabla f(x_k) - \nabla f(x_{k-1}), x_k - x_{k-1} \rangle}, \label{eq:long_bb} \\[4pt]
    \text{(Short BB)} \quad \alpha_k^{\mathrm{BBS}} &= \dfrac{\langle \nabla f(x_k) - \nabla f(x_{k-1}), x_k - x_{k-1} \rangle}{\|\nabla f(x_k) - \nabla f(x_{k-1})\|^2}. \label{eq:short_bb}
\end{align}
Both stepsizes effectively extract pseudo-second-order curvature information from first-order gradient differences, providing a powerful tool to exploit the local geometry structure. 
Convergence analysis of the BB method was initially developed for strongly convex quadratic problems, where $R$-superlinear convergence was established in the two dimensional case \cite{barzilaiTwopointStepSize1988a}, and subsequently extended to $n$ dimensions with a global $R$-linear convergence rate \cite{daiRlinearConvergenceBarzilai2002,raydanBarzilaiBorweinChoice1993}.
However, for general convex functions, the BB method may fail to converge and can even diverge, including in certain strongly convex cases \cite{burdakovStabilizedBarzilaiborweinMethod2019}. 
To address this limitation, early works introduced nonmonotone linesearch techniques \cite{grippoNonmonotoneLineSearch1986} to ensure global convergence \cite{raydanBarzilaiBorweinGradient1997}. Subsequently, considerable effort has been devoted to preserving the efficiency of BB-type stepsizes while avoiding the computational overhead of linesearches. 
This has led to a series of stabilized and theoretically justified adaptive variants, including the stabilized BB method \cite{burdakovStabilizedBarzilaiborweinMethod2019}, the dynamically bounded BB method \cite{wangAdaptiveLearningRate2023}, and the adaptive BB method \cite{zhouAdaBBAdaptiveBarzilaiBorwein2024}, all of which have shown strong practical performance and robustness.

However, the application of these stabilized BB methods remains largely confined to convex settings. In highly nonconvex and unpredictable landscapes, existing bounding strategies and analytical frameworks are often no longer adequate. As a result, the safe incorporation of the aggressive BB stepsize into nonconvex composite optimization, while still preserving rigorous convergence guarantees, remains a challenging and open problem.

\subsection{Our Contributions}

To bridge the gap between the BB stepsizes and nonconvex optimization, we propose a novel adaptive Barzilai-Borwein proximal gradient method for nonconvex optimization (AdaBBNC). Our main contribution is the introduction of a novel adaptive BB-based stepsize that updates using curvature estimates and operates without linesearch. In addition to the original estimate in \eqref{l_k}, we introduce a new curvature estimator tailored to locally nonconvex settings, which combines historical stepsize information with the long BB stepsize \eqref{eq:long_bb}. This design enables AdaBBNC to relax the conservative stepsize truncation commonly imposed by existing methods, thereby allowing larger and more effective steps in ill-conditioned landscapes.

Theoretically, we establish global convergence guarantees for AdaBBNC. By designing a customized parameter-dependent Lyapunov function to accommodate the proposed stepsize, we prove that AdaBBNC achieves an $\mathcal{O}(\epsilon^{-2})$ iteration complexity for finding an $\epsilon$-stationary point. We further evaluate AdaBBNC on several test problems, including the  Rosenbrock function, ill-conditioned matrix factorization, box-constrained quadratic programming, and low-rank matrix completion. The numerical results demonstrate that AdaBBNC is efficient and competitive in comparison with some existing methods, such as the linesearch-free approaches AdaPGNC~\cite{yeSimpleAdaptiveProximal2025a}, AC-PG~\cite{lanProjectedGradientMethods2024a} and AC-PGM~\cite{yagishitaSimpleLinesearchfreeFirstorder2025a}, as well as traditional linesearch-based methods like GD-LS~\cite{armijoMinimizationFunctionsHaving1966}.


To position our work within the recent literature, \Cref{tab:contributions} summarizes the features of AdaBBNC alongside several related gradient-based methods. As shown therein, AdaBBNC is, to the best of our knowledge, the first linesearch-free nonmonotone adaptive framework that successfully incorporates the Barzilai-Borwein stepsize for general nonconvex composite optimization.

\begin{table}[htbp]
    \centering
    \caption{Comparison of the features among related gradient-based methods.}
    \label{tab:contributions}
    \renewcommand{\arraystretch}{1.3}
    \resizebox{\linewidth}{!}{
    \begin{tabular}{lccccc} 
        \toprule
        \textbf{Method} & \textbf{Target Problem} & \textbf{\makecell{Unbounded\\ Domain}} & \textbf{\makecell{Linesearch\\ -Free}} & \textbf{BB stepsize} & \textbf{\makecell{Allows stepsize\\ Growth}} \\ 
        \midrule
        GD-LS \cite{armijoMinimizationFunctionsHaving1966}               &General    & \checkmark & $\times$   & $\times$   & \checkmark \\
        AdGD \cite{malitskyAdaptiveGradientDescent2019a} & Convex & \checkmark & \checkmark & $\times$   & \checkmark \\
        AdaBB \cite{zhouAdaBBAdaptiveBarzilaiBorwein2024}    & Convex & \checkmark & \checkmark & \checkmark & \checkmark \\
        AC-PG \cite{lanProjectedGradientMethods2024a}     & Nonconvex & $\times$ \textsuperscript{a} & \checkmark & $\times$   & $\times$ \textsuperscript{b} \\
        AdaPGNC \cite{yagishitaSimpleLinesearchfreeFirstorder2025a}    & Nonconvex & \checkmark & \checkmark & $\times$   & \checkmark \\
        \rowcolor{gray!15} 
        \textbf{AdaBBNC (Ours)} & \textbf{Nonconvex} & \textbf{\checkmark} & \textbf{\checkmark} & \textbf{\checkmark} & \textbf{\checkmark} \\
        \bottomrule
        \multicolumn{6}{p{\linewidth}}{\footnotesize \textsuperscript{a} AC-PG requires the problem domain to be a bounded compact set.} \\
        \multicolumn{6}{p{\linewidth}}{\footnotesize \textsuperscript{b} AC-PG uses a monotonically nonincreasing stepsize rule, forcing strict shrinkage.} \\
    \end{tabular}
    }
\end{table}

\subsection{Notation and Organization}

\paragraph*{Notation.} Let $\mathbb{R}^n$ be the $n$-dimensional Euclidean space. For any vectors $x, y \in \mathbb{R}^n$, $\langle x, y \rangle$ denotes their standard inner product, $\|x\|$ denotes the Euclidean norm, and $\lceil \cdot \rceil$ denotes the ceiling function. For a proper, closed, and convex function $g: \mathbb{R}^n \rightarrow \mathbb{R} \cup \{+\infty\}$, its subdifferential at a point $x \in \mathbb{R}^n$ is defined as $\partial g(x) = \{v \in \mathbb{R}^n \mid g(y) \ge g(x) + \langle v, y - x \rangle, \; \forall y \in \mathbb{R}^n\}$, and we adopt the convention that $c/0 = +\infty$ for any $c > 0$.

\paragraph*{Organization.}
\Cref{sec:ABBPG} presents the AdaBBNC method and its stepsize update strategy. \Cref{sec:converge} establishes its global convergence. Numerical experiments are reported in \Cref{sec:experiments}. Conclusions are summarized in \Cref{sec:con}.

\section{Adaptive Barzilai-Borwein Proximal Gradient Method}\label{sec:ABBPG}

In this section, we present the adaptive Barzilai–Borwein proximal gradient method for solving \eqref{eq:composite_problem} and begin by introducing several assumptions and definitions.

\begin{assumption}\label{assum:blanket}
Throughout this paper, we make the following standard assumptions regarding the objective function $F(x) = f(x) + g(x)$:
\begin{enumerate}
    \item[{\rm (i)}] The nonsmooth term $g: \mathbb{R}^n \rightarrow \mathbb{R} \cup \{+\infty\}$ is a proper, lower semi-continuous, and convex function.
    \item[{\rm (ii)}] The smooth term $f: \mathbb{R}^n \rightarrow \mathbb{R}$ is continuously differentiable and globally $L$-smooth, i.e., there exists a constant $L > 0$ such that $\|\nabla f(x) - \nabla f(y)\| \le L \|x - y\|$ for all $x, y \in \mathbb{R}^n$.
    \item[{\rm (iii)}] The objective function $F$ is bounded from below, i.e., $F_* := \inf_{x \in \mathbb{R}^n} F(x) > -\infty$.
\end{enumerate}
\end{assumption}


\begin{definition}[Gradient Mapping \cite{nesterovGradientMethodsMinimizing2013}]\label{def:gradient_mapping}
For problem \eqref{eq:composite_problem}, the gradient mapping at $x \in \mathbb{R}^n$ with stepsize $\alpha > 0$ is defined by
    \begin{equation*}
        \mathcal{G}_\alpha(x) := \frac{1}{\alpha} \left( x - \mathrm{prox}_{\alpha g}(x - \alpha \nabla f(x)) \right).
    \end{equation*}
\end{definition}

\Cref{def:gradient_mapping} provides a natural generalization of the gradient to composite objective functions. In particular, when the nonsmooth term vanishes ($g \equiv 0$), it reduces to the standard gradient, i.e., $\mathcal{G}\alpha(x) = \nabla f(x)$. Moreover, for the PGM iterative scheme \eqref{eq:pgm_update}, we have
\begin{equation}\label{eq:gradient_mapping}
        \mathcal{G}_{k} := \mathcal{G}_{\alpha_{k}}(x_{k}) = \frac{x_{k} - x_{k+1}}{\alpha_{k}}.
\end{equation}
    

\begin{definition}[Weak Convexity]\label{def:weak_convexity}
    A differentiable function $f: \mathbb{R}^n \to \mathbb{R}$ is said to be $l$-weakly convex for some constant $l \ge 0$, if the perturbed function $x \mapsto f(x) + \frac{\ell}{2}\|x\|^2$ is convex.
\end{definition}

This geometric property ensures that $f(x)$ can be globally lower-bounded by a quadratic model. Indeed, under \cref{assum:blanket}(ii), there exists a constant $\ell \in [0, L]$ such that $f(x)$ is $l$-weakly convex \cite{davis2019stochastic}, and consequently, the following quadratic bounds hold for all $x, y \in \mathbb{R}^n$ \cite{nesterovIntroductoryLecturesConvex2004}:
\begin{equation}\label{eq:weak_convexity}
    -\frac{\ell}{2} \|x - y\|^2 \le f(x) - f(y) - \langle \nabla f(y), x - y \rangle \le \frac{L}{2} \|x - y\|^2.
\end{equation}

To explore more aggressive stepsizes, we incorporate the BB stepsize into the PGM update \eqref{eq:pgm_update}. A related idea was also presented in \cite[Remark 4]{yeSimpleAdaptiveProximal2025a}, where the authors suggested updating 
$$\alpha_{k}=\min\left\{\sqrt{1+\rho_{k-1}}\alpha_{k-1}, \alpha_k^{\mathrm{BBS}}\right\}$$ 
with $\alpha_k^{\mathrm{BBS}}$ the short BB stepsize in \eqref{eq:short_bb} and ${\rho_k}$ a summable nonnegative sequence. Certainly, we prefer the long BB stepsize $\alpha_k^{\mathrm{BBL}}$ here. In nonconvex settings, $\alpha_k^{\mathrm{BBL}}$ may be zero or negative. Hence, if $\alpha_k^{\mathrm{BBL}}\le0$ or $\ell_k\le0$ (i.e., locally convex case), we keep $\alpha_k = \min\{\sqrt{1+\rho_{k-1}}\alpha_{k-1}, 1/L_k\}$, which has been shown to be effective in convex problems \cite{yeSimpleAdaptiveProximal2025a}. Otherwise, we adopt $\alpha_{k}=\min\left\{\sqrt{1+\rho_{k-1}}\alpha_{k-1}, \alpha_k^{\mathrm{BBL}}\right\}$. Then it follows an adaptive stepsize
\begin{equation}\label{eq:heuristic}
    \alpha_k = 
    \begin{cases}
        \min\left\{\sqrt{1+\rho_{k-1}}\alpha_{k-1}, \frac{1}{L_k}\right\}, & \text{if } \ell_k \le 0 \text{ or } (\alpha_k^{\mathrm{BBL}}) \le 0; \\[4pt]
        \min\left\{\sqrt{1+\rho_{k-1}}\alpha_{k-1}, \alpha_k^{\mathrm{BBL}}\right\}, & \text{otherwise}.
    \end{cases}
\end{equation}
However, the following example indicates that such a direct extension is not feasible.




Consider the problem \eqref{eq:composite_problem} with
\begin{equation}\label{example-adabb}
       f(x) = \begin{cases}
    \frac{1}{4}x^4 - \frac{1}{2}x^2, & \text{if } |x| \le 2, \\[4pt]
    \frac{11}{2}x^2-16|x|+12, & \text{if } |x| > 2
    \end{cases}
    \quad {\rm and} \quad g(x)=0. 
\end{equation}
Clearly, $g(x)$ is convex and continuously differentiable, which is stronger than the requirement in \Cref{assum:blanket}(i). Since $|f''(x)|\le11$, it follows that $f(x)$ is globally $L$-smooth with $L=11$. This, along with $F(x) = f(x) \ge -1/4$, implies that \Cref{assum:blanket} holds. For convenience, we refer to the PGM \eqref{eq:pgm_update} equipped with the adaptive stepsize \eqref{eq:heuristic} as \textbf{Naive-AdaBB}. In this example, we compare Naive-AdaBB with AdaPGNC, setting $x_0 = 0.1$ and $\rho_k = \frac{100(\ln(k+1))^4}{(k+1)^{1.1}},~\forall~k \ge 1$. As shown in \Cref{fig:toy_experiment}, while AdaPGNC takes conservatively small steps to safely traverse the concave region, Naive-AdaBB substantially overestimates the optimal stepsize, resulting in severe numerical instability.



\begin{figure}[htbp]
    \centering
    \includegraphics[width=1.0\linewidth]{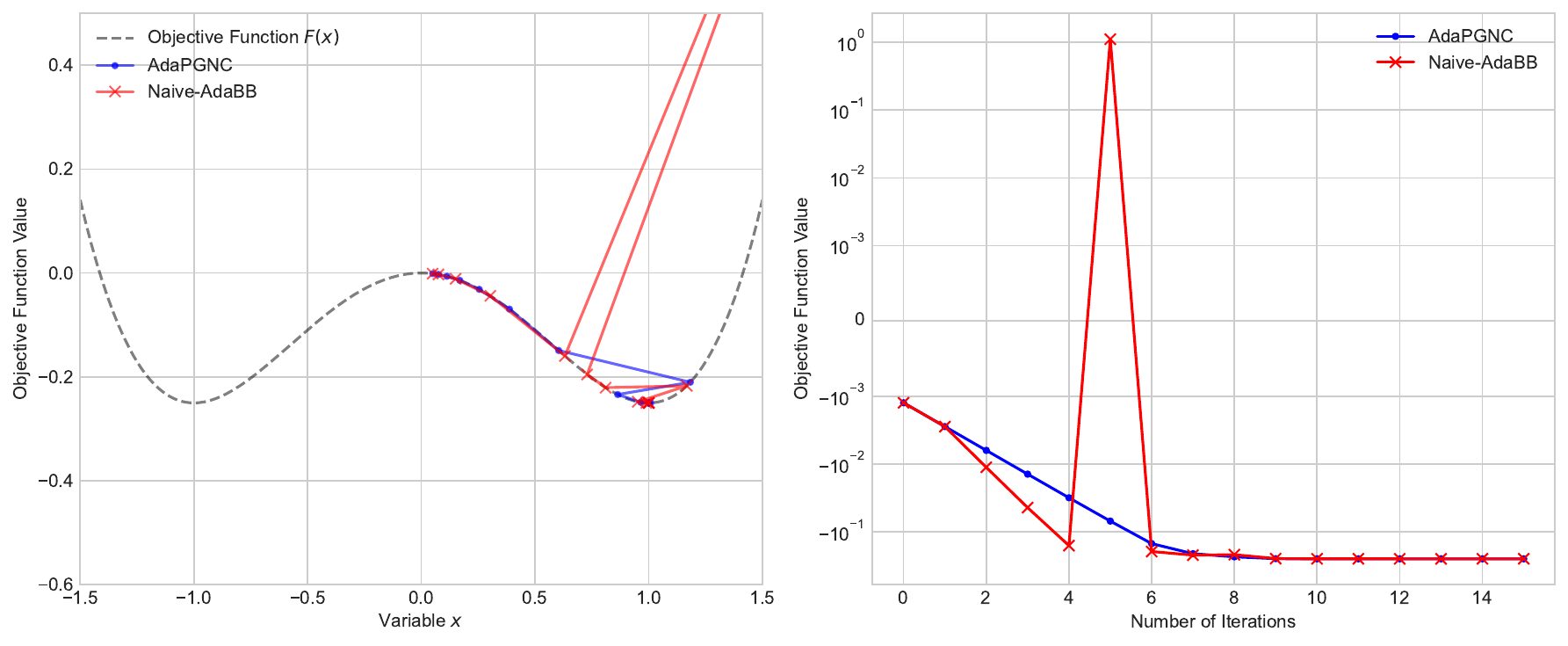}
    \caption{Numerical behavior of AdaPGNC and Naive-AdaBB on Example \eqref{example-adabb}.}
    \label{fig:toy_experiment}
\end{figure}

Therefore, we combine $\alpha_k^{\mathrm{BBL}}$ with $L_k$, $\ell_k$, and $\alpha_{k-1}$, and construct a new curvature approximation:
\begin{equation}\label{eq:generalized_bb_curvature}
    H_k(\tau, \theta) = L_k^2 + \frac{2+\tau}{2\alpha_{k-1}}\ell_k - \frac{\tau+\theta}{\alpha_{k-1}^2} + \frac{\tau}{\alpha_{k-1}} (\alpha_k^{\mathrm{BBL}})^{-1},
\end{equation}
where $\tau \ge 0$ and $\theta \in [0,1)$. Based on this, we propose the following adaptive stepsize:
\begin{equation}\label{eq:stepsize-adbb}
    \alpha_k = 
    \begin{cases}
        \min\left\{\sqrt{1+\rho_{k-1}}\alpha_{k-1}, \frac{1}{L_k}\right\}, & \text{if } \ell_k \le 0 \text{ or } H_k(\tau, \theta) \le 0; \\[4pt]
        \min\left\{\sqrt{1+\rho_{k-1}}\alpha_{k-1}, \frac{1}{\sqrt{H_k(\tau, \theta)}}\right\}, & \text{otherwise}.
    \end{cases}
\end{equation}
Substituting \eqref{eq:stepsize-adbb} into the PGM iterative framework \eqref{eq:pgm_update} leads to an adaptive Barzilai-Borwein proximal gradient method for solving \eqref{eq:composite_problem}, as summarized in \Cref{alg:adabbnc}.


\begin{algorithm}
\caption{Adaptive Barzilai-Borwein Proximal Gradient Method for NonConvex Optimization (AdaBBNC)}
\label{alg:adabbnc}
\begin{algorithmic}[1]
\REQUIRE $x_0 \in \mathbb{R}^n$, initial stepsize $\alpha_0 > 0$, nonnegative sequence $\{\rho_k\}_{k=0}^\infty$ such that $\sum_{k=0}^\infty \rho_k < +\infty$, parameters $\tau \ge 0$ and $\theta \in [0, 1)$.
\STATE Update $x_1 = \mathrm{prox}_{\alpha_0 g}(x_0 - \alpha_0 \nabla f(x_0))$.
\FOR{$k = 1, 2, \dots$}
    \STATE Compute the upper curvature $L_k$ and the lower curvature $\ell_k$ according to \eqref{L_k} and \eqref{l_k}, respectively.
    \IF{$\ell_k \le 0$}
        \STATE $\alpha_k = \min\left\{\sqrt{1+\rho_{k-1}}\alpha_{k-1}, \frac{1}{L_k}\right\}$.
    \ELSE
        \STATE Compute the long BB stepsize $\alpha_k^{\mathrm{BBL}}$ and the generalized BB-curvature bound $H_k(\tau, \theta)$ according to \eqref{eq:long_bb} and \eqref{eq:generalized_bb_curvature}, respectively.
        \IF{$H_k(\tau, \theta) \le 0$}
            \STATE $\alpha_k = \min\left\{\sqrt{1+\rho_{k-1}}\alpha_{k-1}, \frac{1}{L_k}\right\}$.
        \ELSE
            \STATE $\alpha_k = \min\left\{\sqrt{1+\rho_{k-1}}\alpha_{k-1}, \frac{1}{\sqrt{H_k(\tau, \theta)}}\right\}$.
        \ENDIF
    \ENDIF
    \STATE Update $x_{k+1} = \mathrm{prox}_{\alpha_k g}(x_k - \alpha_k \nabla f(x_k))$.
\ENDFOR
\end{algorithmic}
\end{algorithm}

\begin{remark}
    If $\tau=\theta=0$, we have $H_k(0, 0)= L_k^2 + \frac{\ell_k}{\alpha_{k-1}}$. Then AdaBBNC reduces to the method studied in \cite{yeSimpleAdaptiveProximal2025a}. When $\tau+\theta>0$, the negative term $-\frac{\tau+\theta}{\alpha_{k-1}^2}$ becomes active, allowing the stepsize to adapt more aggressively in nonconvex landscapes and potentially leading to larger stepsizes. If extreme curvature causes $H_k(\tau,\theta)\le 0$, AdaBBNC automatically reverts to the conservative selection. Therefore, the proposed AdaBBNC is well-defined.
\end{remark}

\section{Convergence analysis}\label{sec:converge}






In this section, we establish the global convergence of AdaBBNC and first show that the stepsize sequence admits uniform upper and positive lower bounds.

\begin{proposition}[Boundedness of $\{\alpha_k\}$]\label{prop:stepsize_bound}
Let $\{\alpha_k\}$ be the stepsize sequence generated by \cref{alg:adabbnc}, and let $P = \sum_{k=0}^\infty \rho_k < \infty$. Under \cref{assum:blanket}, for any $\tau \ge 0$ and  $\theta \in [0, 1)$, the stepsizes satisfy
\begin{equation*}\label{eq:alpha_bounds}
    0 < \underline{\alpha} \le \alpha_k \le \overline{\alpha}, \quad \forall~k \ge 0,
\end{equation*}
where $\overline{\alpha} = \alpha_0 \exp(\frac{P}{2})$ and
\begin{equation}\label{eq:alpha_lower_bound}
    \underline{\alpha} = 
    \begin{cases}
        \min \left\{ \alpha_0, \frac{1}{L}, \frac{1}{\sqrt{L^2 + \frac{\nu_L^2}{4(\tau+\theta)}}} \right\}, & \text{if } \tau+\theta > 0; \\[10pt]
        \min \left\{ \alpha_0, \frac{\sqrt{5}-1}{2L} \right\}, & \text{if } \tau=\theta=0,
    \end{cases}
\end{equation}
with $\nu_L = \frac{2+\tau}{2}L$.
\end{proposition}

\begin{proof}
    We first derive the upper bound. Using \eqref{eq:stepsize-adbb}, it gives $\alpha_{i+1} \le \sqrt{1+\rho_i} \alpha_i$ for all $i \ge 0$. Since $\ln(1+x) \le x$ for $x \ge 0$, we obtain $\ln \alpha_{i+1} - \ln \alpha_i \le \frac{1}{2}\ln(1+\rho_i) \le \frac{\rho_i}{2}$. Summing from $i=0, \dots, k-1$ yields $\ln \alpha_k - \ln \alpha_0 \le \frac{1}{2} \sum_{i=0}^{k-1} \rho_i \le \frac{P}{2}$. Thus, we have $\alpha_k \le \alpha_0 \exp(\frac{P}{2}) = \overline{\alpha}$.

    Next, we establish the lower bound by mathematical induction. For $k=0$, \eqref{eq:alpha_lower_bound} gives $\alpha_0 \ge \underline{\alpha}$. Suppose that, for some $k \ge 1$, $\alpha_{k-1} \ge \underline{\alpha} > 0$. We now prove that $\alpha_k \ge \underline{\alpha}$.

    It follows from \eqref{L_k}, \eqref{l_k}, and \Cref{assum:blanket} that $L_k \le L$ and $\ell_k \le L$. This along with \eqref{eq:long_bb} and the Cauchy-Schwarz inequality yields leads to
    \begin{equation}\label{eq:bbl_bound}
        (\alpha_k^{\mathrm{BBL}})^{-1} = \frac{\langle \nabla f(x_k) - \nabla f(x_{k-1}), x_k - x_{k-1} \rangle}{\|x_k - x_{k-1}\|^2} \le \frac{\|\nabla f(x_k) - \nabla f(x_{k-1})\|}{\|x_k - x_{k-1}\|} = L_k \le L.
    \end{equation}

    If $\alpha_k = \min\left\{\sqrt{1+\rho_{k-1}}\alpha_{k-1}, \frac{1}{L_k}\right\}$, then, since $\rho_{k-1} \ge 0$, it gives $\sqrt{1+\rho_{k-1}}\alpha_{k-1}\ge\alpha_{k-1}$. By \eqref{eq:alpha_lower_bound} and the fact that $\frac{\sqrt{5}-1}{2} < 1$, we have $\underline{\alpha}\le \frac1L$. This together with the induction hypothesis yields that
    \begin{equation*}
        \alpha_k \ge \min\left\{\alpha_{k-1},\tfrac{1}{L_k}\right\} \ge \min\left\{\underline{\alpha}, \tfrac{1}{L}\right\} = \underline{\alpha}.
    \end{equation*}
    
    If $\alpha_k = \min\left\{\sqrt{1+\rho_{k-1}}\alpha_{k-1}, \frac{1}{\sqrt{H_k(\tau, \theta)}}\right\}$, from \eqref{eq:stepsize-adbb}, it means that $\ell_k>0$ and $H_k(\tau, \theta) >0$. Combined \eqref{l_k} with \eqref{eq:weak_convexity}, $L_k \le L$, $\ell_k \le L$, and \eqref{eq:bbl_bound}, we have
    \begin{equation}\label{eq:combined_curvature_bound}
        \ell_k + 2(\alpha_k^{\mathrm{BBL}})^{-1} = \frac{2\left[f(x_k) - f(x_{k-1}) - \langle \nabla f(x_{k-1}), x_k - x_{k-1} \rangle\right]}{\|x_k - x_{k-1}\|^2} 
        \le \frac{L \|x_k - x_{k-1}\|^2}{\|x_k - x_{k-1}\|^2} = L.
    \end{equation}
    This together with \eqref{eq:generalized_bb_curvature} and $\nu_L = \frac{2+\tau}{2}L$ leads to
    \begin{align}
        H_k(\tau, \theta) &= L_k^2 + \frac{2+\tau}{2\alpha_{k-1}}\ell_k - \frac{\tau+\theta}{\alpha_{k-1}^2} + \frac{\tau (\alpha_k^{\mathrm{BBL}})^{-1}}{\alpha_{k-1}} 
        = L_k^2 + \frac{\ell_k}{\alpha_{k-1}} - \frac{\tau+\theta}{\alpha_{k-1}^2} + \frac{\tau\left(\ell_k + 2(\alpha_k^{\mathrm{BBL}})^{-1}\right)}{2\alpha_{k-1}} \nonumber\\
        &\le L^2 + \frac{L}{\alpha_{k-1}} - \frac{\tau+\theta}{\alpha_{k-1}^2} + \frac{\tau L}{2\alpha_{k-1}} 
        = L^2 + \frac{2+\tau}{2}L\left(\frac{1}{\alpha_{k-1}}\right) - (\tau+\theta)\left(\frac{1}{\alpha_{k-1}}\right)^2\nonumber\\
        &=L^2 + \nu_L\left(\frac{1}{\alpha_{k-1}}\right) - (\tau+\theta)\left(\frac{1}{\alpha_{k-1}}\right)^2.\label{eq:hk_intermediate_bound}
    \end{align}
    
   When $\tau=\theta=0$, we have $\nu_L = L$ and $\underline{\alpha}=\min \left\{ \alpha_0, \frac{\sqrt{5}-1}{2L} \right\}$, which gives $\underline{\alpha}L\le\frac{\sqrt{5}-1}{2}$. This implies that $\underline{\alpha}^2L^2+\underline{\alpha}L\le1$, namely, $L^2 + \frac{L}{\underline{\alpha}}\le\frac{1}{\underline{\alpha}^2}$. Combining this with \eqref{eq:hk_intermediate_bound} and the induction hypothesis yields
    \begin{equation*}
        H_k(0, 0) \le L^2 + \frac{L}{\alpha_{k-1}}
        \le L^2 + \frac{L}{\underline{\alpha}}
        \le \frac{1}{\underline{\alpha}^2}.
    \end{equation*}
    Using \eqref{eq:stepsize-adbb}, $\rho_{k-1}\ge0$ and the induction hypothesis again, it follows that
    $$
    \alpha_k \ge \min\left\{\underline{\alpha}, \tfrac{1}{\sqrt{H_k(0, 0)}}\right\} 
    =\underline{\alpha}.
    $$
    
    When $\tau+\theta > 0$, note that the univariate function $\nu_L x - (\tau+\theta)x^2$ attains its maximum value $\frac{\nu_L^2}{4(\tau+\theta)}$. With \eqref{eq:hk_intermediate_bound}, it leads to
    \begin{equation*}
        H_k(\tau, \theta) \le L^2 + \frac{\nu_L^2}{4(\tau+\theta)}.
    \end{equation*}
    From \eqref{eq:stepsize-adbb}, \eqref{eq:alpha_lower_bound}, $\rho_{k-1}\ge0$ and the induction hypothesis, we get
    \begin{equation*}
        \alpha_k \ge \min\left\{ \alpha_{k-1}, \frac{1}{\sqrt{L^2 + \frac{\nu_L^2}{4(\tau+\theta)}}} \right\} \ge \min\left\{ \underline{\alpha}, \frac{1}{\sqrt{L^2 + \frac{\nu_L^2}{4(\tau+\theta)}}} \right\} = \underline{\alpha}.
    \end{equation*}
Combining the above results and the induction hypothesis, the proof is completed.
\end{proof}

The following result characterizes the relationship between the iterative sequence and the gradient information of the objective function. 

\begin{lemma}\label{lem:lemma1}
Let $\{x_k\}$ be the sequence generated by \cref{alg:adabbnc}. Under \cref{assum:blanket}, there exists a specific subgradient $\xi_k \in \partial F(x_k)$ such that:
\begin{align*}
    \langle \xi_k, x_{k-1} - x_k \rangle &\le F(x_{k-1}) - F(x_k) + \frac{\ell_k}{2} \|x_k - x_{k-1}\|^2, \label{eq:lem1_i} \\
    \langle \xi_k, x_{k-1} - x_k \rangle &= \frac{1}{\alpha_{k-1}} \|x_k - x_{k-1}\|^2 - \langle x_k - x_{k-1}, \nabla f(x_k) - \nabla f(x_{k-1}) \rangle. 
\end{align*}
\end{lemma}

\begin{proof}
    It follows from \eqref{l_k} that
    \begin{equation}\label{eq:proof_lem1_1}
        f(x_{k-1}) = f(x_k) + \langle \nabla f(x_k), x_{k-1} - x_k \rangle - \frac{\ell_k}{2} \|x_k - x_{k-1}\|^2.
    \end{equation}
    Note that $g(x)$ is proper, lower semi-continuous, and convex (\cref{assum:blanket}). For any subgradient $v_k \in \partial g(x_k)$, we have
    \begin{equation}\label{eq:proof_lem1_2}
        g(x_{k-1}) \ge g(x_k) + \langle v_k, x_{k-1} - x_k \rangle.
    \end{equation}
    Since $F(x)=f(x)+g(x)$ and $\xi_k=\nabla f(x_k)+v_k \in \partial F(x_k)$, combining \eqref{eq:proof_lem1_1} and \eqref{eq:proof_lem1_2} yields that
    \begin{equation*}
        F(x_{k-1}) \ge F(x_k) + \langle \xi_k, x_{k-1} - x_k \rangle - \frac{\ell_k}{2} \|x_k - x_{k-1}\|^2,
    \end{equation*}
    which leads to the first result.

    From the first-order optimality condition and the iteration in \Cref{alg:adabbnc}(Step\,15), we get 
    $$x_{k-1} - x_k = \alpha_{k-1} (\nabla f(x_{k-1}) + v_k),$$ where $v_k \in \partial g(x_k)$. Together with $\xi_k = \nabla f(x_k) + v_k$, it gives
    \begin{align*}
        \|x_k - x_{k-1}\|^2 
        &= \langle x_{k-1} - x_k,\, \alpha_{k-1} (\nabla f(x_{k-1}) + v_k) \rangle \nonumber \\
        &= \alpha_{k-1} \langle x_{k-1} - x_k,\, \nabla f(x_k) + v_k + \nabla f(x_{k-1}) - \nabla f(x_k) \rangle \nonumber \\
        &= \alpha_{k-1} \langle x_{k-1} - x_k,\, \xi_k \rangle + \alpha_{k-1} \langle x_{k-1} - x_k,\, \nabla f(x_{k-1}) - \nabla f(x_k) \rangle,
    \end{align*}
    completing the proof since $\alpha_{k-1}>0$.
\end{proof}

We now derive a recursive inequality governing $\|x_{k+1} - x_k\|^2$, which plays a key role in the subsequent convergence analysis.

\begin{lemma}[Generalized Iterate Bound]\label{lem:unified_descent}
Let $\{x_k\}$ and $\{\alpha_k\}$ be the sequences generated by \cref{alg:adabbnc}. Under \cref{assum:blanket}, for any $\tau\ge0$ and $\theta \in [0, 1)$, the following unified bound holds for all $k \ge 1$:
\begin{equation*}\label{eq:lem2_unified}
    \|x_{k+1} - x_k\|^2 \le \alpha_k^2 H_k(\tau, \theta) \|x_k - x_{k-1}\|^2 - (1-\theta)\alpha_k^2 \|\mathcal{G}_{k-1}\|^2 + \frac{(2+\tau)\alpha_k^2}{\alpha_{k-1}}(F(x_{k-1}) - F(x_k)).
\end{equation*}
\end{lemma}

\begin{proof}
   It follows from the descent property of the proximal gradient step (see \cite[Lemma 4]{yagishitaSimpleLinesearchfreeFirstorder2025a}) that
    \begin{equation}\label{eq:lem2_base}
        \|x_{k+1}-x_k\|^2 \le \alpha_k^2 L_k^2 \|x_k - x_{k-1}\|^2 - \alpha_k^2 \|\mathcal{G}_{k-1}\|^2 + \frac{2\alpha_k^2}{\alpha_{k-1}} \left( F(x_{k-1}) - F(x_k) + \frac{\ell_k}{2}\|x_k - x_{k-1}\|^2 \right).
    \end{equation}    
    Using $\theta\in[0,1)$ and \eqref{eq:gradient_mapping}, we obtain
    \begin{equation*}\label{eq:lem2_split}
        \|\mathcal{G}_{k-1}\|^2 
         =(1-\theta) \|\mathcal{G}_{k-1}\|^2 + \theta \|\mathcal{G}_{k-1}\|^2
        = (1-\theta) \|\mathcal{G}_{k-1}\|^2 + \frac{\theta} {\alpha_{k-1}^2} \|x_k - x_{k-1}\|^2.
    \end{equation*}
    Substituting this into \eqref{eq:lem2_base} yields that
        \begin{equation}\label{eq:lem2_base2}
        \begin{array}{ll}
        \|x_{k+1}-x_k\|^2 &\le 
        \alpha_k^2 \left(L_k^2- \frac{\theta} {\alpha_{k-1}^2}\right) \|x_k - x_{k-1}\|^2 
        -(1-\theta)\alpha_k^2 \|\mathcal{G}_{k-1}\|^2
        - \alpha_k^2 \|\mathcal{G}_{k-1}\|^2 \\[6pt]
        &\quad + \frac{2\alpha_k^2}{\alpha_{k-1}} \left( F(x_{k-1}) - F(x_k) + \frac{\ell_k}{2}\|x_k - x_{k-1}\|^2 \right).
        \end{array}
    \end{equation}
    By \cref{lem:lemma1}, we get
    $  \frac{1}{\alpha_{k-1}} \|x_k - x_{k-1}\|^2 - \langle x_k - x_{k-1}, \nabla f(x_k) - \nabla f(x_{k-1}) \rangle =\langle \xi_k, x_{k-1} - x_k \rangle
    \le F(x_{k-1}) - F(x_k) + \frac{\ell_k}{2} \|x_k - x_{k-1}\|^2$. Along with \eqref{eq:long_bb}, it leads to
    \begin{equation}\label{eq:lem2_nonneg}
        0 \le F(x_{k-1}) - F(x_k) + \frac{\ell_k}{2}\|x_k - x_{k-1}\|^2 - \frac{1}{\alpha_{k-1}}\|x_k - x_{k-1}\|^2 + (\alpha_k^{\mathrm{BBL}})^{-1} \|x_k - x_{k-1}\|^2.
    \end{equation}
    Note that $\tau \ge 0$ and $\alpha_k>0$. Multiplying \eqref{eq:lem2_nonneg} by $\frac{\tau \alpha_k^2}{\alpha_{k-1}}$ and adding the resulting inequality to \eqref{eq:lem2_base2} gives
    \begin{align*}
        \|x_{k+1}-x_k\|^2 &\le \alpha_k^2 \left[ L_k^2 + \frac{\ell_k}{\alpha_{k-1}} + \frac{\tau \ell_k}{2\alpha_{k-1}} - \frac{\tau}{\alpha_{k-1}^2} - \frac{\theta}{\alpha_{k-1}^2} + \frac{\tau}{\alpha_{k-1}}(\alpha_k^{\mathrm{BBL}})^{-1} \right] \|x_k - x_{k-1}\|^2 \\
        &\quad - (1-\theta)\alpha_k^2 \|\mathcal{G}_{k-1}\|^2 + \frac{(2+\tau)\alpha_k^2}{\alpha_{k-1}} (F(x_{k-1}) - F(x_k)) \\
        &= \alpha_k^2 \left[ L_k^2 + \frac{2+\tau}{2\alpha_{k-1}}\ell_k - \frac{\tau+\theta}{\alpha_{k-1}^2} + \frac{\tau}{\alpha_{k-1}}(\alpha_k^{\mathrm{BBL}})^{-1} \right] \|x_k - x_{k-1}\|^2 \\
        &\quad - (1-\theta)\alpha_k^2 \|\mathcal{G}_{k-1}\|^2 + \frac{(2+\tau)\alpha_k^2}{\alpha_{k-1}} (F(x_{k-1}) - F(x_k))\\
        &=\alpha_k^2 H_k(\tau, \theta) \|x_k - x_{k-1}\|^2 - (1-\theta)\alpha_k^2 \|\mathcal{G}_{k-1}\|^2 + \frac{(2+\tau) \alpha_k^2}{\alpha_{k-1}} (F(x_{k-1}) - F(x_k)),
    \end{align*}
    where the last equality holds by \eqref{eq:generalized_bb_curvature}. 
\end{proof}




A sufficient condition for bounding $\alpha_k^2 H_k(\tau,\theta)$, which appears as a coefficient in the recursive inequality of \Cref{lem:unified_descent}, is established below.

\begin{lemma}\label{lem:uniform_curvature_bound}
    Let $\{x_k\}$ and $\{\alpha_k\}$ be the sequences generated by \cref{alg:adabbnc}. Under \cref{assum:blanket}, for any $\tau \ge 0$ and $\theta \in [0,1)$ satisfying 
    \begin{equation}\label{eq:tau_upper_bound}
        0 \le \tau \le \frac{2\theta}{[\overline{\alpha} L - 2]_+},
    \end{equation}
    it follows that $\alpha_k^2 H_k(\tau, \theta) \le 1$ for all $k \ge 1$, where $\overline{\alpha}$ is the upper bound of $\alpha_k$ defined in \cref{prop:stepsize_bound}.
\end{lemma}

\begin{proof}
    When $\alpha_k = \min\left\{\sqrt{1+\rho_{k-1}}\alpha_{k-1}, \frac{1}{\sqrt{H_k(\tau, \theta)}}\right\}$, it is clear that $\alpha_k\le \frac{1}{\sqrt{H_k(\tau, \theta)}}$. In this case, the desired result follows immediately. When $\alpha_k = \min\left\{\sqrt{1+\rho_{k-1}}\alpha_{k-1}, \frac{1}{L_k}\right\}$, by \eqref{eq:stepsize-adbb}, the proof proceeds by considering the following two cases separately.

    If $H_k(\tau,\theta) \le 0$, then, noting that $\alpha_k > 0$, the result holds naturally. If $\ell_k \le 0$, with \eqref{eq:generalized_bb_curvature} and \eqref{eq:combined_curvature_bound}, we have
    \begin{equation*}
        H_k(\tau, \theta) = L_k^2 + \frac{\ell_k}{\alpha_{k-1}} - \frac{\tau+\theta}{\alpha_{k-1}^2} + \frac{\tau(\ell_k + 2(\alpha_k^{\mathrm{BBL}})^{-1})}{2\alpha_{k-1}}
        \le L_k^2 - \frac{\tau+\theta}{\alpha_{k-1}^2} + \frac{\tau L}{2\alpha_{k-1}}.
    \end{equation*}
    This along with \cref{prop:stepsize_bound} and the fact that $0< \alpha_k \le 1/L_k$ leads to 
    $$
    \alpha_k^2H_k(\tau, \theta)\le 1-\frac{2(\tau+\theta) - \alpha_{k-1}\tau L}{2\alpha_{k-1}^2L_k^2}
    \le 1-\frac{2(\tau+\theta) -  \overline{\alpha}\tau L}{2\alpha_{k-1}^2L_k^2}
    = 1-\frac{2\theta-(\overline{\alpha}L-2)\tau}{2\alpha_{k-1}^2L_k^2}.
    $$
    Note that \eqref{eq:tau_upper_bound} yields $2\theta-(\overline{\alpha}L-2)\tau \ge0$. Then the result follows immediately.   

\end{proof}

\begin{remark}\label{rem-reasonable}
    We emphasize that the parameter restriction \eqref{eq:tau_upper_bound} is reasonable. Particularly, when $\overline{\alpha}L\le 2$, \eqref{eq:tau_upper_bound} becomes nonrestrictive and reduces to $0\le \tau < +\infty$. Moreover, the condition $\overline{\alpha}L\le 2$ is mild. Indeed, if we choose $\alpha_0 = \frac{1}{L}$ and $\rho_k = \ln 2 \cdot \left(\frac{1}{2}\right)^k$, then $P = \sum_{k=0}^{\infty} \rho_k = 2 \ln 2$ and
    \begin{equation*}
        \overline{\alpha} = \alpha_0 \exp( \frac{P}{2}) = \frac{\exp(\ln 2) }{L} = \frac{2}{L},
    \end{equation*}
    which gives $\overline{\alpha}L=2$.
    \end{remark}

With \cref{lem:unified_descent}, we obtain a crucial inequality for controlling the distance between successive iterates. However, a major difficulty in analyzing nonconvex adaptive methods lies in the fact that the objective sequence ${F(x_k)}$ is not necessarily monotonically decreasing. As done in \cite{yeSimpleAdaptiveProximal2025a}, we construct a Lyapunov function that incorporates the growth-rate regulator $\rho_k$. First, following \cite{yeSimpleAdaptiveProximal2025a}, we define the weighting sequence:
\begin{equation}\label{eq:omega_def}
    \omega_k := \frac{\alpha_k^2}{\alpha_0^2 \prod_{i=1}^k (1+\rho_{i-1}) \prod_{i=1}^k \sqrt{1+\rho_{i-2}}}, \quad \forall~k \ge 1,
\end{equation}
where $\omega_0 = 1$ and $\rho_{-1} = 0$. Since $\alpha_k \ge \underline{\alpha}$ (as established in \cref{prop:stepsize_bound}), $\omega_k$ admits a  positive lower bound:
\begin{equation}\label{omega_bound}
    \omega_k \ge \underline{\omega} := \frac{\underline{\alpha}^2}{\alpha_0^2} \exp\left(-\frac{3P}{2}\right) > 0.
\end{equation}
On this basis, we introduce the following Lyapunov function.

\begin{definition}\label{def:lyapunov}
For any $k \ge 1$, we define an auxiliary sequence $E_k$ and the corresponding Lyapunov function $\mathcal{E}_k$ as
\begin{equation*}
    E_k := F(x_k) - F_* + \frac{\alpha_{k-1}}{(2+\tau)\alpha_k^2} \|x_{k+1} - x_k\|^2, \quad \text{and} \quad \mathcal{E}_k := \omega_k E_k,
\end{equation*}
where $F_* = \inf_x F(x) > -\infty$, $\tau \ge 0$, and $\omega_k$ is defined in \eqref{eq:omega_def}. 
\end{definition}

Note that $F(x_k) \ge F_*$ and $\omega_k > 0$, which implies that $E_k\ge0$ and $\mathcal{E}_k \ge 0$ for all $k \ge 1$. Then we have the following recursive inequality for $\mathcal{E}_k$.

\begin{lemma}\label{lem:lyapunov_descent}
Let $\{x_k\}$ and $\{\alpha_k\}$ be the sequences generated by \cref{alg:adabbnc}. Under \cref{assum:blanket}, for any $\tau \ge 0$ and $\theta \in [0,1)$ satisfying \eqref{eq:tau_upper_bound}, the following descent inequality holds for all $k \ge 1$:
\begin{equation}\label{eq:lyapunov_descent}
    \mathcal{E}_k \le \mathcal{E}_{k-1} -  \frac{1-\theta}{2+\tau} \omega_k\alpha_{k-1} \|\mathcal{G}_{k-1}\|^2,
\end{equation}
where $\rho_{-1} = 0$, and $\mathcal{E}_k$ and $\omega_k$ are defined in \cref{def:lyapunov} and \eqref{eq:omega_def}, respectively.
\end{lemma}

\begin{proof}
    
     Combining \cref{lem:unified_descent} with \cref{lem:uniform_curvature_bound}, we have
    \begin{equation*}
        \|x_{k+1} - x_k\|^2 \le \|x_k - x_{k-1}\|^2 - (1-\theta)\alpha_k^2 \|\mathcal{G}_{k-1}\|^2 + \frac{(2+\tau)\alpha_k^2}{\alpha_{k-1}}(F(x_{k-1}) - F(x_k)).
    \end{equation*}
   Dividing both sides by $\frac{(2+\tau)\alpha_k^2}{\alpha_{k-1}}$ and rearranging the resulting inequality yields
    $$
    F(x_k)-F_*+\frac{\alpha_{k-1}}{(2+\tau)\alpha_k^2}\|x_{k+1} - x_k\|^2 \le F(x_{k-1}) - F_*+\frac{\alpha_{k-1}}{(2+\tau)\alpha_k^2}\|x_k - x_{k-1}\|^2 - \frac{1-\theta}{2+\tau}\alpha_{k-1}\|\mathcal{G}_{k-1}\|^2.
    $$
    By the definition of $E_k$, it leads to
    \begin{equation}\label{eq:base_ek}
        E_k \le F(x_{k-1}) - F_* + \frac{\alpha_{k-1}}{(2+\tau)\alpha_k^2}\|x_k - x_{k-1}\|^2 - \frac{1-\theta}{2+\tau}\alpha_{k-1}\|\mathcal{G}_{k-1}\|^2.
    \end{equation}
    In the following, we further show that 
    \begin{equation}\label{eq:ek_intermediate}
        E_k \le \frac{\alpha_{k-1}^2}{\alpha_k^2}(1+\rho_{k-1})\sqrt{1+\rho_{k-2}}\,E_{k-1} - \frac{1-\theta}{2+\tau}\alpha_{k-1}\|\mathcal{G}_{k-1}\|^2.
    \end{equation}
    by considering two cases based on the value of $\frac{\alpha_{k-1}}{\alpha_k^2}$. 
    
    If $\frac{\alpha_{k-1}}{\alpha_k^2} \le \frac{\alpha_{k-2}}{\alpha_{k-1}^2}$, then it follows from \eqref{eq:base_ek} and \cref{def:lyapunov} that
    \begin{equation*}
            E_k \le  E_{k-1}- \frac{1-\theta}{2+\tau}\alpha_{k-1}\|\mathcal{G}_{k-1}\|^2.
    \end{equation*}
    According to \eqref{eq:stepsize-adbb}, we have $\alpha_k^2 \le \alpha_{k-1}^2(1+\rho_{k-1})$, which implies that $\frac{\alpha_{k-1}^2}{\alpha_k^2}(1+\rho_{k-1}) \ge 1$.     Since $E_{k-1} \ge 0$ and $\rho_{k-2}\ge0$, it gives $E_{k-1}\le \frac{\alpha_{k-1}^2}{\alpha_k^2}(1+\rho_{k-1})\sqrt{1+\rho_{k-2}}$. Along with the above inequalities yields \eqref{eq:ek_intermediate}. 
    
    If $\frac{\alpha_{k-1}}{\alpha_k^2} > \frac{\alpha_{k-2}}{\alpha_{k-1}^2}$, then we have $\frac{\alpha_{k-1}^3}{\alpha_k^2 \alpha_{k-2}} > 1$. Combined with $F(x_{k-1}) - F_* \ge 0$ and \Cref{def:lyapunov}, it yields that
    \begin{align}
       & F(x_{k-1}) - F_* + \frac{\alpha_{k-1}}{(2+\tau)\alpha_k^2}\|x_k - x_{k-1}\|^2
    \le \frac{\alpha_{k-1}^3}{\alpha_k^2 \alpha_{k-2}}(F(x_{k-1}) - F_*) + \frac{\alpha_{k-1}}{(2+\tau)\alpha_k^2}\|x_k - x_{k-1}\|^2\nonumber\\
    & =\frac{\alpha_{k-1}^3}{\alpha_k^2 \alpha_{k-2}} \left( F(x_{k-1}) - F_* + \frac{\alpha_{k-2}}{(2+\tau)\alpha_{k-1}^2}\|x_k - x_{k-1}\|^2 \right)
    = \frac{\alpha_{k-1}^3}{\alpha_k^2 \alpha_{k-2}} E_{k-1}.
    \label{eq:ratio2}
    \end{align}
    Note that \eqref{eq:stepsize-adbb} implies $\frac{\alpha_{k-1}}{\alpha_{k-2}} \le \sqrt{1+\rho_{k-2}}$. This together with $\rho_{k-1}\ge0$ follows that
    $$
    \frac{\alpha_{k-1}^3}{\alpha_k^2 \alpha_{k-2}} 
    \le \frac{\alpha_{k-1}^2}{\alpha_k^2 } \sqrt{1+\rho_{k-2}}
    \le \frac{\alpha_{k-1}^2}{\alpha_k^2 }(1+\rho_{k-1}) \sqrt{1+\rho_{k-2}}.
    $$
Substituting this into \eqref{eq:ratio2} and combining it with \eqref{eq:base_ek}, we obtain \eqref{eq:ek_intermediate}.

Now we complete the proof using the definition of $\omega_k$ in \eqref{eq:omega_def} and the identity $$\frac{\omega_{k-1}}{\omega_k} = \frac{\alpha_{k-1}^2}{\alpha_k^2}(1+\rho_{k-1})\sqrt{1+\rho_{k-2}}.$$
Multiplying \eqref{eq:ek_intermediate} by $\omega_k$ and substituting $\mathcal{E}_k = \omega_k E_k$ yields \eqref{eq:lyapunov_descent}.
\end{proof}

We complete the analysis by presenting the following global convergence result.

\begin{theorem}\label{thm:convergence_rate1}
Let $\{x_k\}$ and $\{\alpha_k\}$ be the sequences generated by \cref{alg:adabbnc}. Under \cref{assum:blanket}, for any $\tau \ge 0$ and $\theta \in [0,1)$ satisfying \eqref{eq:tau_upper_bound}, we have $\lim_{k \to \infty} \|\mathcal{G}_{\overline{\alpha}}(x_k)\| = 0$, and the following bound on the gradient mapping norm holds for any integer $N \ge 1$:
    \begin{equation*}
        \min_{1 \le k \le N} \|\mathcal{G}_{\overline{\alpha}}(x_{k-1})\|^2 \le \min_{1 \le k \le N} \|\mathcal{G}_{k-1}\|^2 \le \frac{(2+\tau)\mathcal{E}_0}{(1-\theta)\underline{\omega} \, \underline{\alpha} N}.
    \end{equation*}
Here $\overline{\alpha}$ and $\underline{\alpha}$ are the stepsize bounds established in \cref{prop:stepsize_bound}, and $\underline{\omega} > 0$ is the lower bound of the weighting sequence in \eqref{omega_bound}. 

Moreover, for any $\epsilon > 0$, the number of iterations required to generate a point $\hat{x}$ satisfying $\|\mathcal{G}_{\overline{\alpha}}(\hat{x})\| \le \epsilon$ is at most $\lceil\frac{(2+\tau)\mathcal{E}_0}{(1-\theta)\underline{\omega} \, \underline{\alpha} \epsilon^2}\rceil$, which is of order $\mathcal{O}(\epsilon^{-2})$.

\end{theorem}

\begin{proof}
    For any integer $N \ge 1$, summing \eqref{eq:lyapunov_descent} from $k=1$ to $N$ and using $\mathcal{E}_N \ge 0$, we obtain
    \begin{equation*}
        \frac{1-\theta}{2+\tau} \sum_{k=1}^{N} \omega_k \alpha_{k-1} \|\mathcal{G}_{k-1}\|^2 \le \mathcal{E}_0 - \mathcal{E}_N \le \mathcal{E}_0.
    \end{equation*}
    This along with \Cref{prop:stepsize_bound} and \eqref{omega_bound} leads to
    \begin{equation}\label{eq:sum_grad_mapping}
        \frac{(1-\theta)\underline{\omega} \, \underline{\alpha}}{2+\tau} \sum_{k=1}^{N} \|\mathcal{G}_{k-1}\|^2 \le \mathcal{E}_0.
    \end{equation}
    Note that $\underline{\omega} > 0$, $\underline{\alpha} > 0$, $\tau \ge 0$, and $\theta \in [0,1)$. Taking $N \to \infty$ in \eqref{eq:sum_grad_mapping} yields $\sum_{k=1}^{\infty} \|\mathcal{G}_{k-1}\|^2 < \infty$, which implies $\lim\limits_{k \to \infty} \|\mathcal{G}_{k-1}\| = 0$.
By \cref{def:gradient_mapping}, the norm $\|\mathcal{G}_{\eta}(x)\|$ is nonincreasing with respect to $\eta > 0$. Since $\alpha_{k-1} \le \overline{\alpha}$ holds for all $k \ge 1$ (see \cref{prop:stepsize_bound}), we have
    \begin{equation}\label{eq:grad_mapping_mono}
        0 \le \|\mathcal{G}_{\overline{\alpha}}(x_{k-1})\| \le \|\mathcal{G}_{\alpha_{k-1}}(x_{k-1})\| = \|\mathcal{G}_{k-1}\|.
    \end{equation}
    Combining this with $\lim\limits_{k \to \infty} \|\mathcal{G}_{k-1}\| = 0$, it follows immediately that $\lim\limits_{k \to \infty} \|\mathcal{G}_{\overline{\alpha}}(x_{k-1})\| = 0$.
    
    Furthermore, using \eqref{eq:sum_grad_mapping} and \eqref{eq:grad_mapping_mono}, we deduce
    \begin{equation*}
        N \min_{1 \le k \le N} \|\mathcal{G}_{\overline{\alpha}}(x_{k-1})\|^2 \le N \min_{1 \le k \le N} \|\mathcal{G}_{k-1}\|^2 \le \sum_{k=1}^{N} \|\mathcal{G}_{k-1}\|^2 \le \frac{(2+\tau)\mathcal{E}_0}{(1-\theta)\underline{\omega} \, \underline{\alpha}},
    \end{equation*}
    which leads to the desired result as $N>0$. 
\end{proof}

\begin{remark} \label{remark3}

Although we argued in \Cref{rem-reasonable} that the convergence condition \eqref{eq:tau_upper_bound} is reasonable, it remains rather conservative. Indeed, in our implementation (see \Cref{sec:experiments} below), we set $\alpha_0=0.01$ and $\rho_0=10^{10}$. It follows from \cref{prop:stepsize_bound} that $\bar{\alpha}=0.01\exp(\frac{10^{10}}{2})\approx10^{2\times 10^9}$. Combining this with \eqref{eq:tau_upper_bound} and the fact that $\theta\in[0,1)$ implies that the admissible range of $\tau$ becomes extremely small, effectively forcing $\tau$ to be close to $0$. Nevertheless, the results in \Cref{sec:ablation} show that AdaBBNC remains effective even for relatively large values of $\tau$ (e.g., $\tau=100$). Therefore, narrowing the gap between this conservative theoretical requirement and the observed empirical behavior represents an interesting direction for future research.


\end{remark}

\begin{remark} \label{remark4}

However, an excessively large value of $\tau$ may cause $H_k(\tau,\theta)$ to become nonpositive, i.e., $H_k(\tau,\theta)\le 0$. In this case, the stepsize in \cref{alg:adabbnc} reduces to $\alpha_k = \min\left\{\sqrt{1+\rho_{k-1}}\alpha_{k-1}, 1/L_k\right\}$, thereby failing to exploit the BB stepsize. To clarify this effect, we rewrite $H_k$ in \eqref{eq:generalized_bb_curvature} as follows:
\begin{equation*}
    H_k(\tau, \theta) = \tau \underbrace{\left( \frac{\ell_k}{2\alpha_{k-1}} + \frac{(\alpha_k^{BBL})^{-1}}{\alpha_{k-1}} - \frac{1}{\alpha_{k-1}^2} \right)}_{:= \Delta_k} + \left( L_k^2 + \frac{\ell_k}{\alpha_{k-1}} - \frac{\theta}{\alpha_{k-1}^2} \right).
\end{equation*}
Clearly, if $\Delta_k < 0$, then a sufficiently large $\tau$ will lead to $H_k(\tau,\theta) \le 0$.  This situation can indeed occur in practice. For example, consider $f(x) = -x^2$ and $g(x) \equiv 0$. For any $x_k \neq x_{k-1}$, it follows from \eqref{L_k} and \eqref{eq:long_bb} that
\begin{equation*}
   L_k = \frac{|-2(x_k - x_{k-1})|}{|x_k - x_{k-1}|} = 2, \quad (\alpha_k^{BBL})^{-1} = \frac{\langle -2(x_k - x_{k-1}), x_k - x_{k-1} \rangle}{|x_k - x_{k-1}|^2} = -2. 
\end{equation*}
Combining the definition of $\ell_k$ in \eqref{l_k} leads to
\begin{equation*}
    \ell_k = \frac{2(-x_k^2 - (-x_{k-1}^2) - (-2x_k)(x_k - x_{k-1}))}{|x_k - x_{k-1}|^2} = \frac{2(x_k - x_{k-1})^2}{|x_k - x_{k-1}|^2} = 2.
\end{equation*}
Then we have
\begin{equation*}
    \Delta_k = \frac{2}{2\alpha_{k-1}} + \frac{-2}{\alpha_{k-1}} - \frac{1}{\alpha_{k-1}^2} = -\frac{1}{\alpha_{k-1}} - \frac{1}{\alpha_{k-1}^2}.
\end{equation*}
Since $\alpha_{k-1} > 0$, it gives $\Delta_k < 0$, which in turn implies that $\lim\limits_{\tau \to +\infty} H_k(\tau,\theta) = -\infty$. Consequently, AdaBBNC fully discards the BB stepsize, which is consistent with the observations reported in \Cref{sec:ablation}.

\end{remark}

\section{Numerical Experiments}
\label{sec:experiments}

In this section, we conduct a series of numerical experiments to evaluate the performance of the proposed AdaBBNC. All methods are implemented in Python and run on a MacBook Air equipped with an Apple M4 chip and 16GB of RAM. 

To provide a comprehensive assessment, our evaluation is divided into two main parts: a primary performance evaluation across various nonconvex optimization tasks (\cref{sec:main_experiments}) and a sensitivity analysis of the parameters $\tau$ and $\theta$ (\cref{sec:ablation}).

The methods compared in our experiments are summarized in \Cref{tab:parameter_settings}. For GD-LS, we set $\beta=1/2$. For adaptive proximal methods, the choice of the stepsize growth regulator $\rho_k$ plays a crucial role. We adopt the following specific sequence in our numerical experiments:
\begin{equation}\label{eq:rho_selection}
    \rho_k = \frac{100(\ln(k+1))^4}{(k+1)^{1.1}}, \quad \forall~k \ge 1,
\end{equation}
with the initial value set to $\rho_0 = 10^{10}$. This particular formulation was originally proposed by Hoai et al. \cite{hoaiaCompositeOptimizationModels2010, hoaiNovelStepsizeGradient2024} to control stepsize expansion and was subsequently adopted as a primary setting in the AdaPGNC framework \cite{yagishitaSimpleLinesearchfreeFirstorder2025a}. To ensure a rigorous and fair evaluation, we empirically identify \eqref{eq:rho_selection} as the most robust configuration of $\rho_k$ for AdaPGNC across the evaluated tasks. We therefore applied the same sequence for AdaBBNC. In addition, the parameters in AdaBBNC are fixed to $\tau = 1$ and $\theta = 0.1$ across all main experiments, unless otherwise specified in the subsequent parameter sensitivity analysis. To accommodate the distinct characteristics of each optimization task, detailed parameter settings, including the initial stepsize and stopping criteria, are provided in the respective subsections.

\begin{table}[htbp]
    \centering
    \caption{Comparison of the stepsize updating rules among the evaluated methods.}
    \label{tab:parameter_settings}
    \renewcommand{\arraystretch}{1.8} 
    
    \begin{threeparttable}
        \begin{tabular}{lp{11.5cm}}
            \toprule
            \textbf{Method} & \textbf{Parameters and Stepsize Updating Rule} \\
            \midrule
            
            GD-LS & 
            $\alpha_k = \alpha_0 \beta^{m_k}$, where $m_k \ge 0$ is the smallest integer satisfying the Armijo linesearch condition. \\
            
            \midrule
            
            AC-PG / AC-PGM\tnote{*} & 
            $\alpha_k = \frac{1}{\alpha \gamma_k}$, where $\gamma_k = \max\{L_0, \dots, L_{k-1}\}$. This ensures $\alpha_k$ is monotonically nonincreasing. \\
            
            \midrule
            
            AdaPGNC & 
            $\alpha_k = \begin{cases} 
            \min\left\{\sqrt{1+\rho_{k-1}}\alpha_{k-1}, \frac{1}{L_k}\right\}, & \text{if } \ell_k \le 0; \\ 
            \min\left\{\sqrt{1+\rho_{k-1}}\alpha_{k-1}, \frac{1}{\sqrt{2}L_k}, \sqrt{\frac{\alpha_{k-1}}{2\ell_k}}\right\}, & \text{otherwise}. 
            \end{cases}$ \\
            
            \midrule
            
            \textbf{AdaBBNC} & 
            $\alpha_k = \begin{cases} 
            \min\left\{\sqrt{1+\rho_{k-1}}\alpha_{k-1}, \frac{1}{L_k}\right\}, & \text{if } \ell_k \le 0 \text{ or } H_k(\tau,\theta) \le 0; \\ 
            \min\left\{\sqrt{1+\rho_{k-1}}\alpha_{k-1}, \frac{1}{\sqrt{H_k(\tau,\theta)}}\right\}, & \text{otherwise}. 
            \end{cases}$ 
             \\
            \bottomrule
        \end{tabular}
        
        \begin{tablenotes}
            \footnotesize
            \item[*] \textbf{Note:} AC-PG and AC-PGM share the same stepsize framework based on the historical maximum curvature $\gamma_k$. Their differences lie in the choice of parameters and the problem settings: AC-PG fixes $\alpha = 1$ and is designed for constrained optimization, whereas AC-PGM requires $\alpha > 1$ and is developed for composite optimization problems.
        \end{tablenotes}
    \end{threeparttable}
\end{table}

\subsection{Performance Evaluation}
\label{sec:main_experiments}

This subsection is dedicated to assessing the convergence speed, computational efficiency, and stepsize dynamics of AdaBBNC against established baselines. To systematically examine its behavior, we conduct experiments across four progressively challenging tasks.

We begin by visualizing the method’s adaptivity in navigating pathological valley landscapes using the classic  Rosenbrock function. We then assess its robustness to ill-conditioning by following benchmark settings from recent studies \cite{yagishitaSimpleLinesearchfreeFirstorder2025a, lanProjectedGradientMethods2024a} to construct highly ill-conditioned synthetic instances for symmetric matrix factorization and box-constrained quadratic programming. Finally, we extend the evaluation to real-world scenarios by considering the low-rank matrix completion problem on the widely-used MovieLens datasets..

\subsubsection{Rosenbrock Function}

Before scaling up to high-dimensional machine learning tasks, we first illustrate the adaptivity of AdaBBNC using the classical   Rosenbrock function \cite{rosenbrock1960automatic}, defined as follows:
\begin{equation}
    f(x) = (1 - x_1)^2 + 100(x_2 - x_1^2)^2,
\end{equation}
which features a narrow, highly ill-conditioned valley with the global minimum at $x^* = [1, 1]^\top$, serving as an ideal stress test for stepsize stability.

For both AdaBBNC and AdaPGNC, the initial stepsize is set to $\alpha_0 = 0.002$. All methods are terminated when $\|x_k - x^*\|_2 < 10^{-6}$.

\begin{figure}[htbp]
    \centering
    \includegraphics[width=1.0\textwidth]{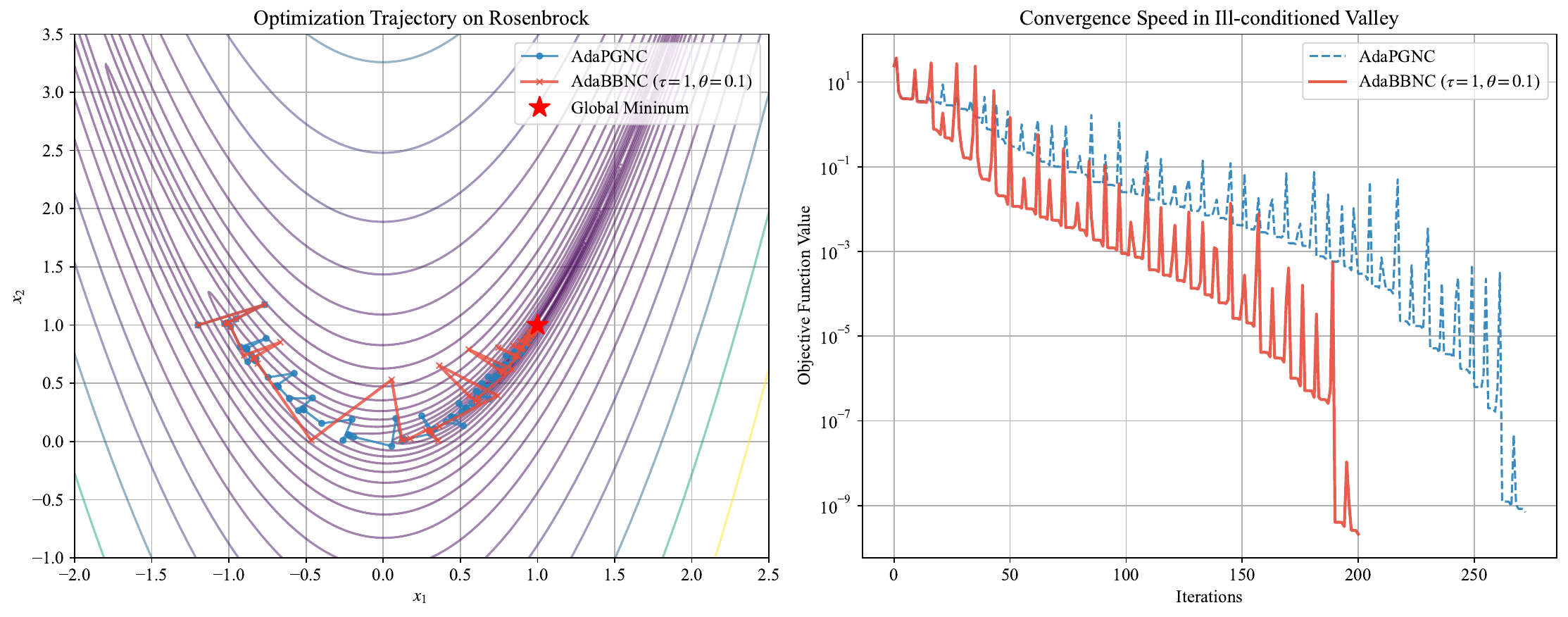} 
    \caption{Optimization dynamics on the  Rosenbrock function. The left panel visualizes the optimization trajectories navigating the highly ill-conditioned narrow valley. The right panel compares the convergence speed, depicting the nonmonotone decrease of the objective value over iterations.}
    \label{fig:rosenbrock}
\end{figure}

As shown in \Cref{fig:rosenbrock}, the ill-conditioned narrow valley of the  Rosenbrock function forces AdaPGNC to take conservative steps. In contrast, the proposed curvature approximation $H_k(\tau, \theta)$ allows AdaBBNC to safely execute much larger steps. Consequently, AdaBBNC achieves a much faster nonmonotone objective descent, satisfying the high-precision stopping criterion in approximately $210$ iterations, compared to roughly $270$ iterations for AdaPGNC. This confirms that the proposed curvature estimation facilitates more effective stepsize selection and significantly improves convergence speed.

\subsubsection{Ill-Conditioned Matrix Factorization}
\label{sec:results_matrix}

We then consider the symmetric matrix factorization problem \cite{chi2019nonconvex}: 
\begin{equation} \label{eq:sym_mf}
    \min_{X \in \mathbb{R}^{n \times r}}~\frac{1}{4} \|XX^\top - M^*\|_F^2,
\end{equation}
where $X \in \mathbb{R}^{n \times r}$ is the low-rank factor matrix, $r$ is the target rank, and $M^* \in \mathbb{R}^{n \times n}$ is the target symmetric positive semi-definite matrix. To introduce severe ill-conditioning, we construct $M^*$ with a condition number of 1000, setting its eigenvalues to decay from $10^1$ to $10^{-2}$. Furthermore, the initial point $X_0$ is drawn from a standard Gaussian distribution scaled by $10^{-6}$, placing it extremely close to the origin.

The initial stepsizes for AdaBBNC, AdaPGNC, and GD-LS are uniformly set to $\alpha_0 = 10^{-2}$, while AC-PGM employs an adaptation factor of $\alpha=1.1$ with an initial Lipschitz estimate $L_0=100$. The optimization terminates when the gradient norm falls below $10^{-6}$ or the iteration budget of $3000$ is reached. All runtimes are measured in CPU process time.

As illustrated in \Cref{fig:results_matrix}, AdaBBNC effectively handles severe ill-conditioning, achieving  faster convergence than other methods. To further assess computational efficiency and scalability, we extend the experiments to larger dimensions ranging from $n = 2000$ to $4000$. The quantitative results, summarized in \Cref{tab:matrix_results}, confirm that our proposed method consistently maintains superior performance across other approaches.

\begin{figure}[htbp]
    \centering
    \includegraphics[width=0.75\textwidth]{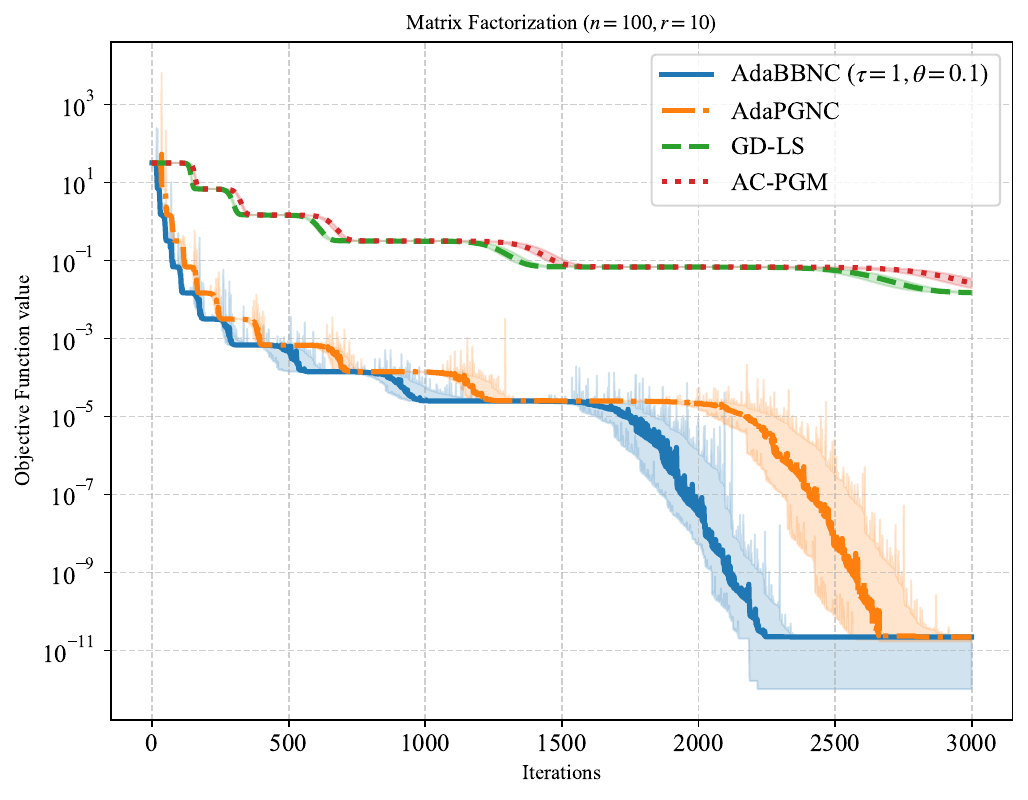}
    \caption{Convergence behavior for the ill-conditioned symmetric matrix factorization problem.}
    \label{fig:results_matrix}
\end{figure}

\begin{table}[htbp]
    \centering
    \caption{Comparison on the ill-conditioned symmetric matrix factorization problem. }
    \label{tab:matrix_results}
    \begin{tabular}{clcccc}
        \toprule
        \textbf{\makecell{Dimensions
        \\ $(n, r)$}} & \textbf{Method} & \textbf{Iterations} & \textbf{GradRes} & \textbf{Loss} & \textbf{Time (s)} \\
        \midrule

        $(2000, 20)$  & \textbf{AdaBBNC}  & \textbf{2789.0} & $\mathbf{4.82 \times 10^{-6}}$ & $\mathbf{5.38 \times 10^{-10}}$ & \textbf{155.58} \\
                      & AdaPGNC                   & 2888.7          & $1.48 \times 10^{-4}$          & $2.53 \times 10^{-7}$           & 160.68 \\
                      & AC-PGM                    & 3000.0          & $2.78 \times 10^{-2}$          & $6.85 \times 10^{-2}$           & 166.88 \\
                      & GD-LS                     & 3000.0          & $6.74 \times 10^{-2}$          & $6.25 \times 10^{-2}$           & 249.16 \\
        \midrule
        $(3000, 20)$  & \textbf{AdaBBNC}  & \textbf{2843.0} & $\mathbf{3.17 \times 10^{-4}}$ & $\mathbf{4.05 \times 10^{-6}}$  & \textbf{385.36} \\
                      & AdaPGNC                   & 2877.0          & $6.46 \times 10^{-4}$          & $4.10 \times 10^{-6}$           & 419.02 \\
                      & AC-PGM                    & 3000.0          & $2.99 \times 10^{-2}$          & $6.84 \times 10^{-2}$           & 451.34 \\
                      & GD-LS                     & 3000.0          & $7.10 \times 10^{-2}$          & $6.15 \times 10^{-2}$           & 658.28 \\
        \midrule
        $(3000, 30)$  & \textbf{AdaBBNC}  & \textbf{2751.0} & $\mathbf{1.91 \times 10^{-6}}$ & $\mathbf{8.36 \times 10^{-11}}$ & \textbf{379.21} \\
                      & AdaPGNC                   & 2922.7          & $1.39 \times 10^{-5}$          & $1.20 \times 10^{-8}$           & 402.70 \\
                      & AC-PGM                    & 3000.0          & $9.38 \times 10^{-2}$          & $9.03 \times 10^{-2}$           & 417.60 \\
                      & GD-LS                     & 3000.0          & $5.59 \times 10^{-2}$          & $7.91 \times 10^{-2}$           & 600.15 \\
        \midrule
       
        $(4000, 30)$  & \textbf{AdaBBNC}  & \textbf{2712.0} & $\mathbf{9.44 \times 10^{-7}}$ & $\mathbf{2.06 \times 10^{-11}}$ & \textbf{608.19} \\
                      & AdaPGNC                   & 2821.3          & $3.12 \times 10^{-6}$          & $1.67 \times 10^{-10}$          & 632.16 \\
                      & AC-PGM                    & 3000.0          & $9.27 \times 10^{-2}$          & $9.00 \times 10^{-2}$           & 671.67 \\
                      & GD-LS                     & 3000.0          & $5.81 \times 10^{-2}$          & $7.87 \times 10^{-2}$           & 1000.45 \\
        \bottomrule
    \end{tabular}
\end{table}

\subsubsection{Box-Constrained QP}
\label{sec:box_qp}

The box-constrained quadratic programming (Box-QP) problem \cite{birgin2000nonmonotone} is defined as follows:
\begin{equation}
    \begin{aligned}
        \min_{x \in \mathbb{R}^n} \quad & \frac{1}{2} x^\top Q x + c^\top x \\
        \text{s.t.} \quad & x_i \in [a_i, b_i], \quad \forall i = 1, \dots, n,
    \end{aligned}
\end{equation}
where $Q \in \mathbb{R}^{n \times n}$ is an ill-conditioned symmetric matrix containing negative eigenvalues, and the box constraint is given by $[a_i, b_i] = [-2, 2]$ for all $i = 1, \dots, n$. 

 To rigorously assess the adaptability of the methods, the Hessian matrix $Q$ is generated with a maximum eigenvalue of $L=10.0$ and $10\%$ negative eigenvalues set to $-0.1$. The condition number $\kappa$ is controlled by allowing the remaining positive eigenvalues to decay logarithmically from $1.0$ down to $10.0 / \kappa$. Furthermore, to ensure a heavily constrained optimal solution, the linear term is set as $c = -Q x_{\text{unc}}$, where the unconstrained stationary point is sampled as $x_{\text{unc}} \sim \mathcal{N}(0, 4I)$. The initial point $x_0$ is uniformly sampled from $[-1, 1]^n$.

In our numerical experiments, the initial stepsize for both AdaBBNC and AdaPGNC is uniformly set to $\alpha_0 = 10^{-2}$. To ensure a comprehensive evaluation of the auto-conditioned rule, AC-PG is tested with different initial Lipschitz estimates $L_0 \in \{0.1L, 0.5L\}$ (which implicitly determines its initial stepsize), as recommended in \cite{lanProjectedGradientMethods2024a}. All competing methods are executed under a fixed budget of $3000$ iterations.

We first analyze a highly ill-conditioned setting with condition number $\kappa = 10^4$, as illustrated in \Cref{fig:box_qp_dynamics}. The empirical results indicate that both AdaBBNC and AdaPGNC exhibit better performance compared to AC-PG, which clearly demonstrates the advantage of the nonmonotonic stepsize strategy in navigating severe ill-conditioning. Subsequently, we evaluate the methods across a broader spectrum of ill-conditioning by varying $\kappa \in \{10^2, 10^3, 10^4, 10^5\}$. As summarized in \Cref{tab:sensitivity_qp}, our proposed method consistently outperforms the other compared approaches.

\begin{figure}[htbp]
    \centering
    \includegraphics[width=1.0\textwidth]{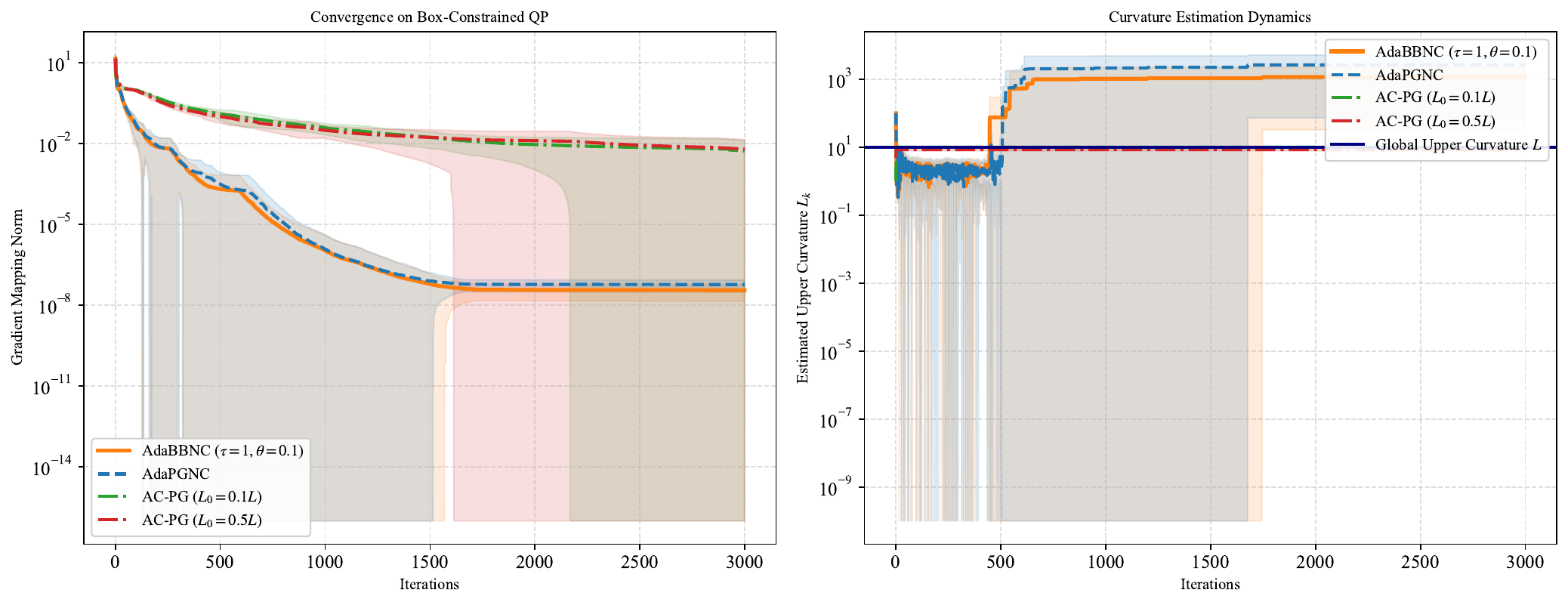}
    \caption{Convergence behavior and curvature estimation dynamics on the Box-QP problem with $\kappa=10^4$. The left panel tracks the projected gradient mapping norm $\|\mathcal{G}_k\|$ over iterations. The right panel depicts the corresponding estimated local upper curvature $L_k$ utilized by each approach.}
    \label{fig:box_qp_dynamics}
\end{figure}

\begin{table}[htbp]
    \centering
    \caption{Final norm of the gradient mapping $\|\mathcal{G}_k\|$ under different values of $\kappa$ after $3000$ iterations.}
    \label{tab:sensitivity_qp}
    \renewcommand{\arraystretch}{1.2}
    \begin{tabular}{l c c c c}
        \toprule
        \textbf{Method} & $\mathbf{\kappa = 10^2}$ & $\mathbf{\kappa = 10^3}$ & $\mathbf{\kappa = 10^4}$ & $\mathbf{\kappa = 10^5}$ \\
        \midrule
        AdaPGNC & $9.17 \times 10^{-8}$ & $7.45 \times 10^{-8}$ & $5.85 \times 10^{-8}$ & $4.40 \times 10^{-8}$ \\
        \textbf{AdaBBNC} & $\mathbf{5.30 \times 10^{-8}}$ & $\mathbf{3.67 \times 10^{-8}}$ & $\mathbf{3.62 \times 10^{-8}}$ & $\mathbf{2.48 \times 10^{-8}}$ \\
        AC-PG ($L_0=0.1L$) & $3.24 \times 10^{-3}$ & $2.65 \times 10^{-3}$ & $5.39 \times 10^{-3}$ & $4.76 \times 10^{-3}$ \\
        AC-PG ($L_0=0.5L$) & $2.90 \times 10^{-3}$ & $2.73 \times 10^{-3}$ & $6.52 \times 10^{-3}$ & $2.48 \times 10^{-3}$ \\
        
        \bottomrule
    \end{tabular}
\end{table}

\subsubsection{Low-Rank Matrix Completion}
\label{sec:matrix_completion}

Finally, we evaluate the methods on the low-rank matrix completion problem \cite{wenSurveyNonconvexRegularizationbased2018}. Given a set of observed entries $\Omega$ of size $N$, where each element $(i, j, s) \in \Omega$ indicates that the $(i, j)$-th entry of the target matrix has a value of $s$, the goal is to recover a rank-$r$ approximation by factorizing the target matrix into two low-rank matrices $U \in \mathbb{R}^{p \times r}$ and $V \in \mathbb{R}^{q \times r}$. To enhance numerical stability and training efficiency, we formulate the optimization problem by fitting the rating residuals:
\begin{equation}
    \min_{Z := (U,V) \in \mathbb{R}^{p \times r} \times \mathbb{R}^{q \times r}} ~ \frac{1}{2N} \sum_{(i,j,s) \in \Omega} \left( (UV^\top)_{ij} - (s - \mu) \right)^2 + \frac{1}{2N} \|U^\top U - V^\top V\|_F^2,
\end{equation}
where $\mu = \frac{1}{N} \sum_{(i,j,s) \in \Omega} s$ represents the global mean of all observed ratings. The regularization term $\|U^\top U - V^\top V\|_F^2$ is introduced to ensure scaling consistency between the factor matrices $U$ and $V$.

\begin{table}[htbp]
\centering
\caption{Key statistics and configurations of the MovieLens benchmark datasets.}
\label{tab:dataset_stats}
\begin{tabular}{l|c|c}
\toprule
\textbf{Dataset Characteristics} & \textbf{MovieLens-100K} & \textbf{MovieLens-1M} \\
\midrule
Matrix Rank ($r$)                & $100$                   & $500$                 \\
Dimension ($p$)                  & $943$                   & $6,040$               \\
Dimension ($q$)                  & $1,682$                 & $3,900$               \\
Total Ratings ($N$)              & $100,000$               & $1,000,209$           \\
Optimization Variables ($(p+q)r$) & $262,500$               & $4,970,000$           \\
\bottomrule
\end{tabular}
\end{table}

To assess the performance of AdaBBNC on real-world applications, we conducted experiments using the MovieLens-100K and MovieLens-1M collaborative filtering datasets \cite{harper2015movielens}. The detailed dataset characteristics and experimental configurations are reported in \Cref{tab:dataset_stats}.

For the experimental setup, the target rank is set to $r=100$ for MovieLens-100K and $r=500$ for the larger-scale MovieLens-1M dataset. The factor matrices $U$ and $V$ are initialized by sampling from the Gaussian distribution $\mathcal{N}(0,1/r)$. To ensure a fair comparison, all tested methods are initialized with $\alpha_0=10^{-2}$ and terminated according to a unified stopping criterion, with the maximum CPU time capped at $2000$ seconds.

\Cref{fig:completion} presents the numerical results. The two adaptive methods (AdaBBNC and AdaPGNC) converge consistently on both datasets, whereas GD-LS and AC-PGM exhibit stagnation behavior. In particular, AdaBBNC achieves the best overall performance across both the MovieLens-100K and MovieLens-1M datasets.

\begin{figure}[htbp]
    \centering
    \includegraphics[width=0.95\textwidth]{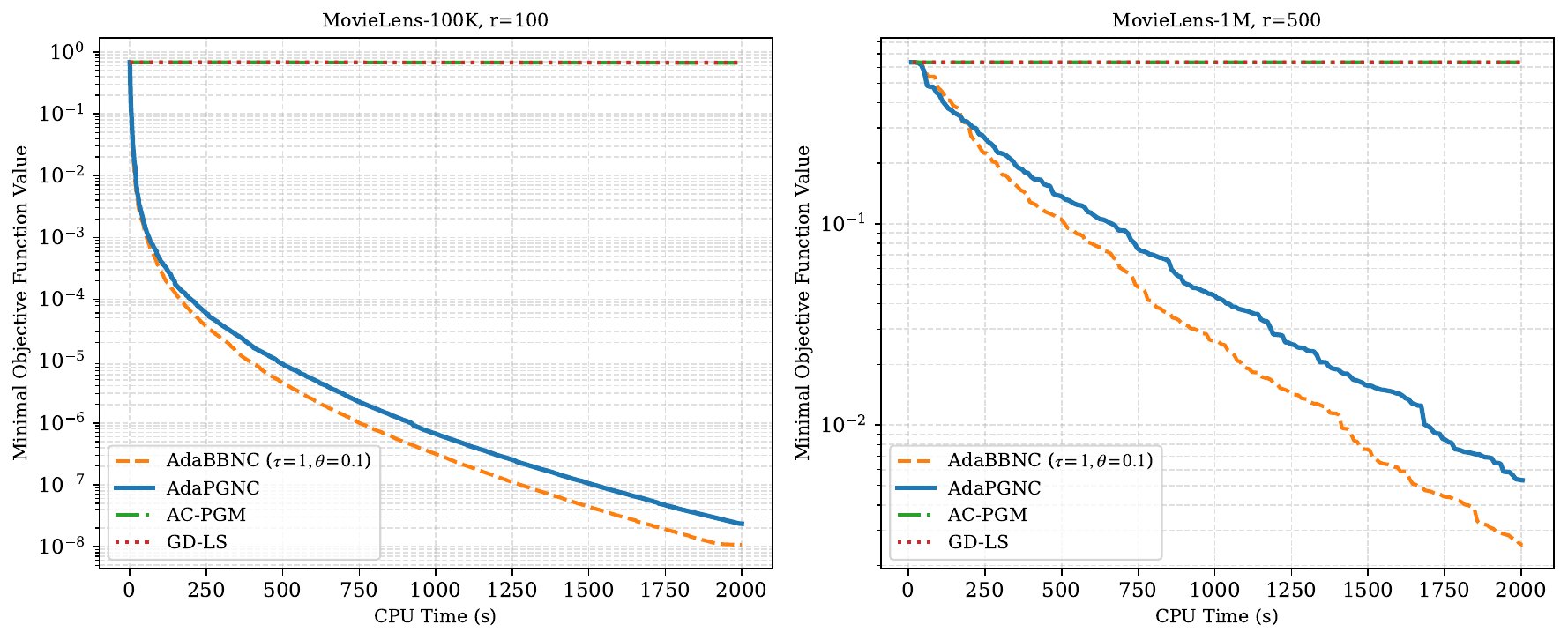}
    \caption{Convergence of the minimal objective function value with respect to CPU time on the low-rank matrix completion problem. The left and right panels present the optimization results for the MovieLens-100K ($r=100$) and MovieLens-1M ($r=500$) datasets, respectively.}
    \label{fig:completion}
\end{figure}


\subsection{ Sensitivity Analysis on Parameters \texorpdfstring{$\tau$}{tau} and \texorpdfstring{$\theta$}{theta}}
\label{sec:ablation}

In this subsection, we present a parameter sensitivity analysis to examine the effects of $\tau$ and $\theta$ on the performance of AdaBBNC. These parameters are introduced to regulate the curvature estimate $H_k(\tau,\theta)$ defined in \eqref{eq:generalized_bb_curvature}.

To rigorously evaluate the sensitivity of AdaBBNC to the parameters $\tau$ and $\theta$, we consider the ill-conditioned symmetric matrix factorization problem introduced in \cref{sec:results_matrix}, with dimensions set to $n=4000$ and $r=30$. All methods are run until the gradient norm falls below $10^{-6}$ or a maximum of $3000$ iterations is reached. The number of iterations required for convergence and the final objective value are averaged over three independent random seeds. For comparison, we include AdaPGNC and a degenerate variant of AdaBBNC with $\tau=0$ and $\theta=0$ as baseline methods in \Cref{fig:theta_sensitivity,fig:tau_sensitivity}. 

\begin{figure}[htbp]
    \centering
    \includegraphics[width=1.0\textwidth]{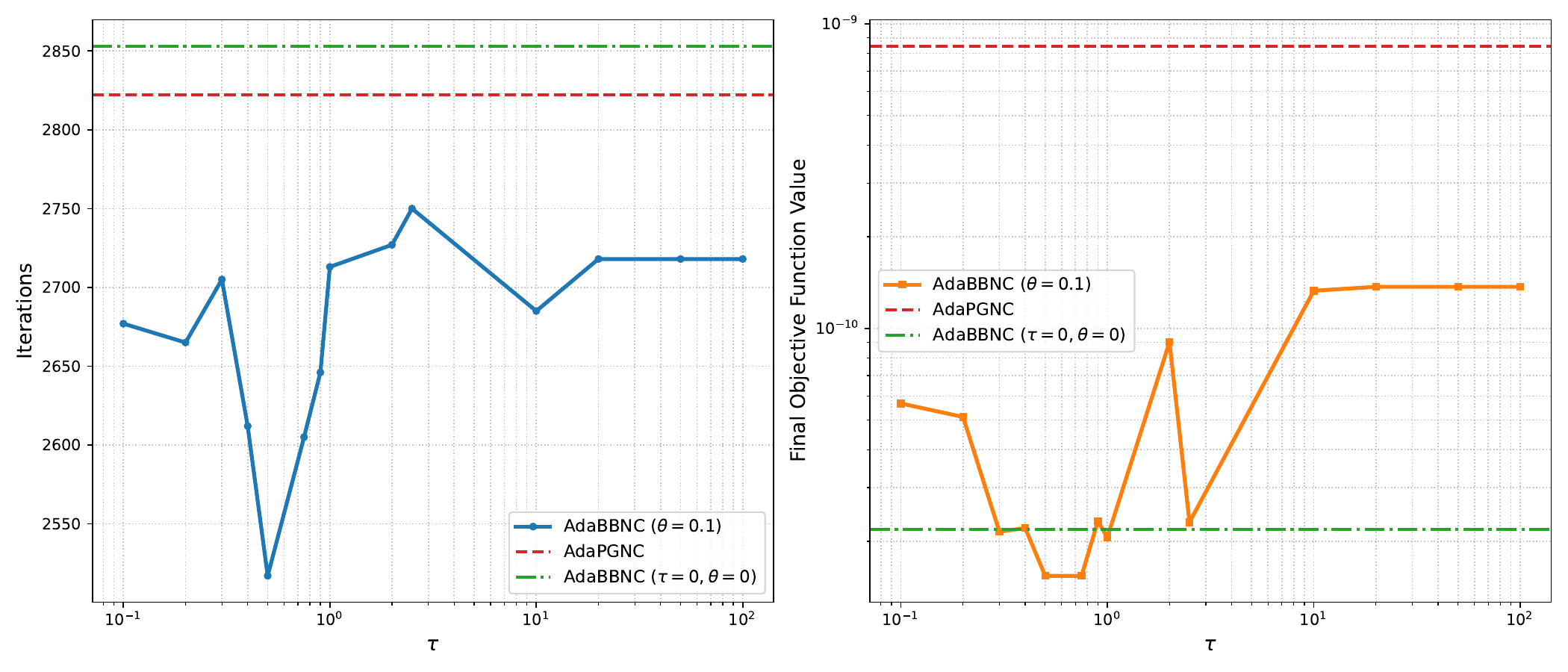}
    \caption{Sensitivity analysis of AdaBBNC with respect to the parameter $\tau$ on the ill-conditioned matrix factorization problem ($n=4000$, $r=30$). The left panel shows the number of iterations required for convergence, and the right panel presents the final objective value.}
    \label{fig:tau_sensitivity}
\end{figure}

First, we investigate the influence of the parameter $\tau$. To this end, we fix $\theta = 0.1$ and evaluate $\tau$ over a broad range from $0.1$ to $100$. The corresponding optimization results are presented in \cref{fig:tau_sensitivity}. As shown, all tested values of $\tau$ require fewer iterations to converge than both baseline methods: AdaPGNC, which requires 2822 iterations and attains a final loss of $8.45 \times 10^{-10}$, and the degenerate AdaBBNC ($\tau=0$, $\theta=0$), which requires 2853 iterations with a final loss of $2.19 \times 10^{-11}$.

Taking both convergence efficiency and the final objective value into account, AdaBBNC empirically achieves its best performance when $\tau$ lies in the range $[0.1,1.0]$. For example, setting $\tau = 0.5$ yields the fastest convergence, requiring only 2517 iterations while attaining a highly accurate objective value of $1.54 \times 10^{-11}$.

Furthermore, we observe that when $\tau$ becomes excessively large ($\tau \geq 20$), the optimization behavior stabilizes at nearly identical outcomes, as illustrated in \cref{fig:tau_sensitivity}. Specifically, for $\tau \in \left\{20,50,100\right\}$, the method consistently terminates after exactly 2718 iterations and achieves the same final loss of $1.37 \times 10^{-10}$. By examining the internal optimization dynamics, we find that for these large values of $\tau$, the curvature estimator $H_k$ remains nonpositive throughout the entire optimization process, resulting in a $100\%$ rejection rate of the BB stepsize. This empirical observation is fully consistent with the theoretical analysis in \cref{remark4}.

\begin{figure}[htbp]
    \centering
    \includegraphics[width=1.0\textwidth]{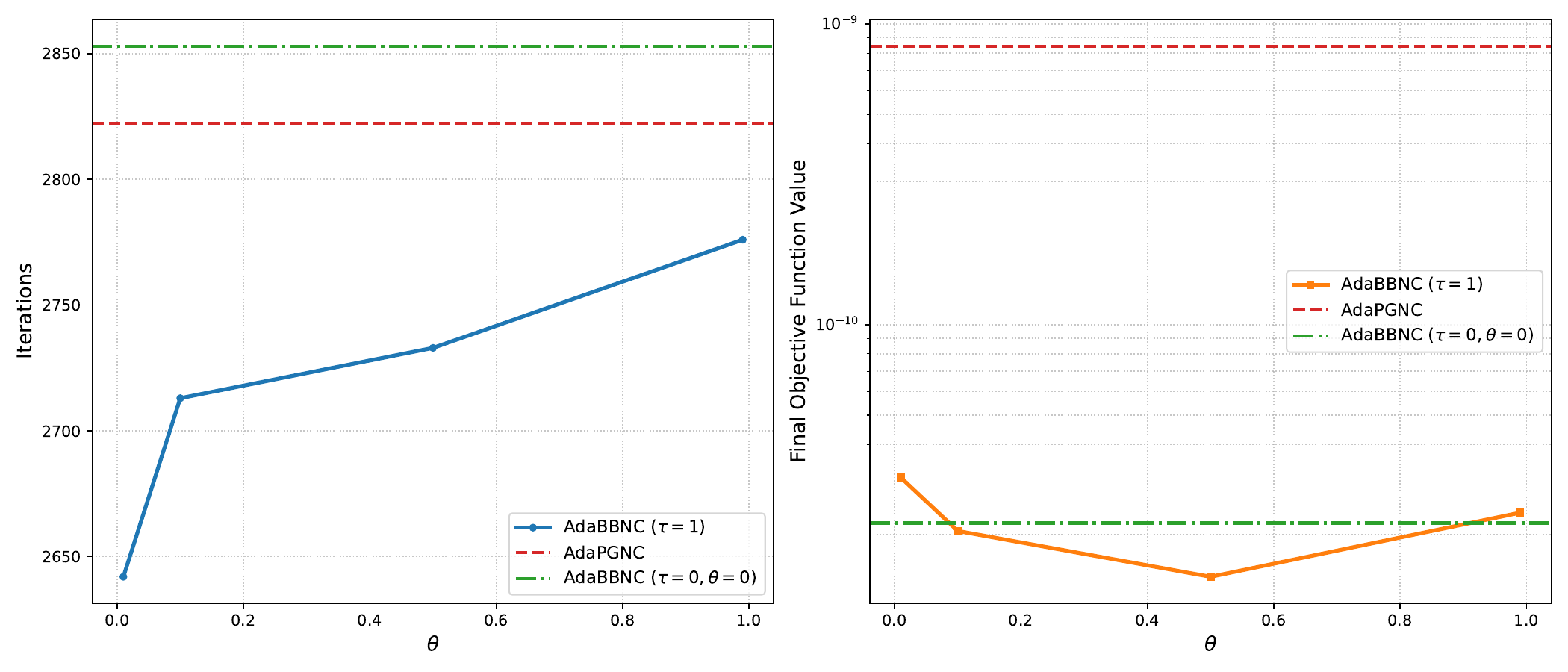}
    \caption{Sensitivity analysis of AdaBBNC with respect to the parameter $\theta$ on the ill-conditioned matrix factorization problem ($n=4000$, $r=30$). The left panel shows the number of iterations required for convergence, and the right panel presents the final objective value.}
    \label{fig:theta_sensitivity}
\end{figure}

Next, we fix $\tau = 1$ and consider $\theta \in \left\{0.01, 0.1, 0.5, 0.99\right\}$. As shown in \cref{fig:theta_sensitivity}, all tested values of $\theta$ require fewer iterations than both baseline methods. Specifically, AdaPGNC requires 2822 iterations and attains a final loss of $8.45 \times 10^{-10}$, and the degenerate AdaBBNC ($\tau=0$, $\theta=0$) requires 2853 iterations.

Furthermore, we observe that a very small value of $\theta$, namely $\theta = 0.01$, yields the fastest convergence, requiring only 2642 iterations, but leads to a relatively less accurate final objective value of $3.10 \times 10^{-11}$. In contrast, as $\theta$ approaches $1$ (e.g., $\theta = 0.99$), the method becomes more conservative, requiring 2776 iterations without providing additional gains in accuracy, plateauing at a final loss of $2.37 \times 10^{-11}$. Therefore, considering both convergence efficiency and the final objective value, choosing $\theta$ excessively close to either $0$ or $1$ is generally not recommended. Instead, moderate values of $\theta$ provide a better balance between convergence speed and numerical accuracy.


\section{Conclusion}\label{sec:con}

In this paper, we proposed the adaptive Barzilai-Borwein proximal gradient method (\Cref{alg:adabbnc}) for nonconvex optimization for solving nonconvex optimization problem \eqref{eq:composite_problem}. To alleviate the conservative stepsize behavior commonly observed in existing linesearch-free methods under locally nonconvex geometries, we introduced a flexible parameterized curvature estimate $H_k(\tau, \theta)$. On this basis, the BB stepsize was successfully incorporated into the proximal gradient framework. We further proved that AdaBBNC attains the optimal $\mathcal{O}(\epsilon^{-2})$ iteration complexity for computing an $\epsilon$-stationary point under reasonable assumptions.

Extensive numerical experiments on several representative nonconvex optimization tasks demonstrated the effectiveness of the proposed method. In particular, AdaBBNC exhibited strong robustness in ill-conditioned nonconvex settings and consistently achieved competitive convergence performance relative to existing approaches.

Future work will focus on extending the proposed adaptive BB-curvature framework to broader classes of large-scale optimization problems, particularly in modern machine learning applications, where adaptive curvature exploitation may further improve optimization efficiency and scalability.

\bibliographystyle{plainnat} 
\bibliography{nonconvex1}

@article{armijoMinimizationFunctionsHaving1966,
  title = {Minimization of Functions Having {{Lipschitz}} Continuous First Partial Derivatives},
  author = {Armijo, Larry},
  year = 1966,
  month = jan,
  journal = {Pacific Journal of Mathematics},
  volume = {16},
  number = {1},
  pages = {1--3},
  doi = {10.2140/pjm.1966.16.1},
  urldate = {2026-03-23},
  langid = {english}
}

@article{barzilaiTwopointStepSize1988a,
  title = {Two-Point Step Size Gradient Methods},
  author = {Barzilai, Jonathan and Borwein, Jonathan M.},
  journal = {IMA Journal of Numerical Analysis},
  volume = {8},
  number = {1},
  pages = {141--148},
  year = {1988},
  publisher = {Oxford University Press},
  doi = {10.1093/imajna/8.1.141}
}

@book{beckFirstOrderMethodsOptimization2017,
  title = {First-{{Order Methods}} in {{Optimization}}},
  author = {Beck, Amir},
  year = 2017,
  month = oct,
  publisher = {{Society for Industrial and Applied Mathematics}},
  address = {Philadelphia, PA},
  doi = {10.1137/1.9781611974997},
  urldate = {2026-03-25},
  langid = {english}
}

@article{beckSmoothingFirstOrder2012,
  title = {Smoothing and First Order Methods: {{A}} Unified Framework},
  author = {Beck, Amir and Teboulle, Marc},
  year = 2012,
  journal = {SIAM Journal on Optimization},
  volume = {22},
  number = {2},
  pages = {557--580},
  publisher = {SIAM}
}

@article{bellocruzConvergenceForwardBackward2016,
  title = {On the Convergence of the Forward--Backward Splitting Method with Linesearches},
  author = {Bello Cruz, Jos{\'e} Yunier and Nghia, Tran T.A.},
  year = 2016,
  month = nov,
  journal = {Optimization Methods and Software},
  volume = {31},
  number = {6},
  pages = {1209--1238},
  doi = {10.1080/10556788.2016.1214959},
  urldate = {2026-03-24},
  langid = {english},
}

@article{bertsekas1997nonlinear,
  title = {Nonlinear Programming},
  author = {Bertsekas, Dimitri P.},
  year = 1997,
  journal = {Journal of the Operational Research Society},
  volume = {48},
  number = {3},
  pages = {334--334},
  publisher = {Taylor \& Francis},
  doi = {10.1057/palgrave.jors.2600425}
}

@article{birgin2000nonmonotone,
  title = {Nonmonotone Spectral Projected Gradient Methods on Convex Sets},
  author = {Birgin, Ernesto G. and Mart{\'\i}nez, Jos{\'e} Mario and Raydan, Marcos},
  year = 2000,
  journal = {SIAM Journal on Optimization},
  volume = {10},
  number = {4},
  pages = {1196--1211},
  publisher = {SIAM}
}

@article{burdakovStabilizedBarzilaiborweinMethod2019,
  title = {Stabilized Barzilai-Borwein Method},
  author = {Burdakov, Oleg and Dai, Yu-Hong and Huang, Na},
  year = 2019,
  journal = {Journal of Computational Mathematics},
  pages = {916--936},
  publisher = {JSTOR}
}

@article{candes2008introduction,
  title = {An Introduction to Compressive Sampling},
  author = {Cand{\`e}s, Emmanuel J. and Wakin, Michael B.},
  journal = {IEEE Signal Processing Magazine},
  volume = {25},
  number = {2},
  pages = {21--30},
  year = {2008},
  publisher = {IEEE},
  doi = {10.1109/MSP.2007.914731}
}

@article{candesExactMatrixCompletion2012,
  title = {Exact Matrix Completion via Convex Optimization},
  author = {Cand{\`e}s, Emmanuel J. and Recht, Benjamin},
  year = 2012,
  journal = {Communications of the ACM},
  volume = {55},
  number = {6},
  pages = {111--119},
  publisher = {ACM New York, NY, USA}
}

@article{chi2019nonconvex,
  title = {Nonconvex Optimization Meets Low-Rank Matrix Factorization: {{An}} Overview},
  author = {Chi, Yuejie and Lu, Yue M. and Chen, Yuxin},
  year = 2019,
  journal = {IEEE Transactions on Signal Processing},
  volume = {67},
  number = {20},
  pages = {5239--5269},
  publisher = {IEEE}
}

@incollection{combettesProximalSplittingMethods2011,
  title = {Proximal Splitting Methods in Signal Processing},
  booktitle = {Fixed-Point Algorithms for Inverse Problems in Science and Engineering},
  author = {Combettes, Patrick L. and Pesquet, Jean-Christophe},
  year = 2011,
  pages = {185--212},
  publisher = {Springer}
}

@article{daiRlinearConvergenceBarzilai2002,
  title = {R-Linear Convergence of the {{Barzilai}} and {{Borwein}} Gradient Method},
  author = {Dai, Yu-Hong and Liao, Li-Zhi},
  year = 2002,
  journal = {IMA Journal of numerical analysis},
  volume = {22},
  number = {1},
  pages = {1--10},
  publisher = {Oxford University Press}
}

@article{davis2019stochastic,
  title = {Stochastic Model-Based Minimization of Weakly Convex Functions},
  author = {Davis, Damek and Drusvyatskiy, Dmitriy},
  year = 2019,
  journal = {SIAM Journal on Optimization},
  volume = {29},
  number = {1},
  pages = {207--239},
  publisher = {SIAM}
}

@article{goldsteinCauchysMethodMinimization1962,
  title = {Cauchy's Method of Minimization},
  author = {Goldstein, A. A.},
  year = 1962,
  month = dec,
  journal = {Numerische Mathematik},
  volume = {4},
  number = {1},
  pages = {146--150},
  doi = {10.1007/BF01386306}
}

@article{grippoNonmonotoneLineSearch1986,
  title = {A Nonmonotone Line Search Technique for {{Newton}}'s Method},
  author = {Grippo, Luigi and Lampariello, Francesco and Lucidi, Stefano},
  year = 1986,
  journal = {SIAM Journal on Numerical Analysis},
  volume = {23},
  number = {4},
  pages = {707--716},
  publisher = {SIAM}
}

@article{harper2015movielens,
  title = {The Movielens Datasets: {{History}} and Context},
  author = {Harper, F. Maxwell and Konstan, Joseph A.},
  year = 2015,
  journal = {Acm Transactions on Interactive Intelligent Systems (tiis)},
  volume = {5},
  number = {4},
  pages = {1--19},
  publisher = {Acm New York, NY, USA}
}

@book{hastieStatisticalLearningSparsity2015,
  title = {Statistical Learning with Sparsity: The Lasso and Generalizations},
  author = {Hastie, Trevor and Tibshirani, Robert and Wainwright, Martin},
  series = {Monographs on Statistics and Applied Probability},
  volume = {143},
  year = {2015},
  publisher = {CRC Press},
  address = {Boca Raton, FL}
}

@article{hoaiaCompositeOptimizationModels2010,
  title = {Composite Optimization Models via Proximal Gradient Method with a Novel Enhanced Adaptive Stepsize},
  author = {Hoaia, Pham Thi and Thaia, Nguyen Pham Duy},
  year = 2010,
  journal = {Optimization Online, preprint, https://optimization-online. org/2010/08/2714}
}

@article{hoaiNovelStepsizeGradient2024,
  title = {A Novel Stepsize for Gradient Descent Method},
  author = {Hoai, Pham Thi and Vinh, Nguyen The and Chung, Nguyen Phung Hai},
  year = 2024,
  month = mar,
  journal = {Operations Research Letters},
  volume = {53},
  pages = {107072},
  doi = {10.1016/j.orl.2024.107072},
  urldate = {2026-01-23},
  langid = {english}
}

@article{lanOptimalParameterfreeGradient2026,
  title = {Optimal and Parameter-Free Gradient Minimization Methods for Convex and Nonconvex Optimization},
  author = {Lan, Guanghui and Ouyang, Yuyuan and Zhang, Zhe},
  year = 2026,
  journal = {Mathematical Programming},
  pages = {1--40},
  publisher = {Springer}
}

@article{lanProjectedGradientMethods2024a,
  title = {Projected Gradient Methods for Nonconvex and Stochastic Optimization: New Complexities and Auto-Conditioned Stepsizes},
  author = {Lan, Guanghui and Li, Tianjiao and Xu, Yangyang},
  year = 2024,
  journal = {arXiv preprint arXiv:2412.14291},
  eprint = {2412.14291},
  archiveprefix = {arXiv}
}

@article{latafatAdaptiveProximalAlgorithms2023,
  title = {Adaptive Proximal Algorithms for Convex Optimization under Local {{Lipschitz}} Continuity of the Gradient},
  author = {Latafat, Puya and Themelis, Andreas and Stella, Lorenzo and Patrinos, Panagiotis},
  year = 2023,
  month = jan,
  journal = {Mathematical Programming},
  volume = {213},
  number = {1-2},
  pages = {433--471},
  doi = {10.1007/s10107-024-02143-7},
  urldate = {2025-11-24},
  langid = {english}
}

@article{leeProximalNewtontypeMethods2014,
  title = {Proximal {{Newton-type}} Methods for Minimizing Composite Functions},
  author = {Lee, Jason D. and Sun, Yuekai and Saunders, Michael A.},
  year = 2014,
  journal = {SIAM Journal on Optimization},
  volume = {24},
  number = {3},
  pages = {1420--1443},
  publisher = {SIAM}
}

@article{lionsSplittingAlgorithmsSum1979,
  title = {Splitting Algorithms for the Sum of Two Nonlinear Operators},
  author = {Lions, Pierre-Louis and Mercier, Bertrand},
  year = 1979,
  journal = {SIAM Journal on Numerical Analysis},
  volume = {16},
  number = {6},
  pages = {964--979},
  publisher = {SIAM}
}

@article{liuNonmonotoneAcceleratedProximal2024,
  title = {A Nonmonotone Accelerated Proximal Gradient Method with Variable Stepsize Strategy for Nonsmooth and Nonconvex Minimization Problems},
  author = {Liu, Hongwei and Wang, Ting and Liu, Zexian},
  year = 2024,
  month = aug,
  journal = {Journal of Global Optimization},
  volume = {89},
  number = {4},
  pages = {863--897},
  doi = {10.1007/s10898-024-01366-4},
  urldate = {2025-12-30},
  langid = {english}
}

@article{malitskyAdaptiveGradientDescent2019a,
  title = {Adaptive Gradient Descent without Descent},
  author = {Malitsky, Yura and Mishchenko, Konstantin},
  year = 2019,
  journal = {arXiv preprint arXiv:1910.09529},
  eprint = {1910.09529},
  archiveprefix = {arXiv}
}

@article{malitskyAdaptiveProximalGradient2024a,
  title = {Adaptive Proximal Gradient Method for Convex Optimization},
  author = {Malitsky, Yura and Mishchenko, Konstantin},
  year = 2024,
  journal = {Advances in Neural Information Processing Systems},
  volume = {37},
  pages = {100670--100697}
}

@article{malitskyGoldenRatioAlgorithms2020,
  title = {Golden Ratio Algorithms for Variational Inequalities},
  author = {Malitsky, Yura},
  year = 2020,
  month = nov,
  journal = {Mathematical Programming},
  volume = {184},
  number = {1-2},
  pages = {383--410},
  doi = {10.1007/s10107-019-01416-w},
  urldate = {2026-02-04},
  langid = {english}
}

@article{malitskyProjectedReflectedGradient2015,
  title = {Projected {{Reflected Gradient Methods}} for {{Monotone Variational Inequalities}}},
  author = {Malitsky, Yura},
  year = 2015,
  month = jan,
  journal = {SIAM Journal on Optimization},
  volume = {25},
  number = {1},
  pages = {502--520},
  doi = {10.1137/14097238X},
  urldate = {2026-02-09}
}

@article{neal2011distributed,
  title={Distributed optimization and statistical learning via the alternating direction method of multipliers},
  author={Boyd, Stephen and Parikh, Neal and Chu, Eric and Peleato, Borja and Eckstein, Jonathan},
  journal={Foundations and Trends in Machine Learning},
  volume={3},
  number={1},
  pages={1--122},
  year={2011},
  publisher={Now Publishers, Inc.}
}

@article{nesterovGradientMethodsMinimizing2013,
  title = {Gradient Methods for Minimizing Composite Functions},
  author = {Nesterov, Yurii},
  year = 2013,
  journal = {Mathematical programming},
  volume = {140},
  number = {1},
  pages = {125--161},
  publisher = {Springer}
}

@book{nesterovIntroductoryLecturesConvex2004,
  title = {Introductory {{Lectures}} on {{Convex Optimization}}},
  author = {Nesterov, Yurii},
  year = 2004,
  series = {Applied {{Optimization}}},
  volume = {87},
  publisher = {Springer US},
  address = {Boston, MA},
  doi = {10.1007/978-1-4419-8853-9},
  urldate = {2026-03-31},
  copyright = {http://www.springer.com/tdm},
  langid = {english}
}

@article{nesterovSmoothMinimizationNonsmooth2005,
  title = {Smooth Minimization of Non-Smooth Functions},
  author = {Nesterov, Yurii},
  year = 2005,
  journal = {Mathematical programming},
  volume = {103},
  number = {1},
  pages = {127--152},
  publisher = {Springer}
}

@article{parikhProximalAlgorithms2014,
  title = {Proximal Algorithms},
  author = {Parikh, Neal and Boyd, Stephen},
  journal = {Foundations and Trends in Optimization},
  volume = {1},
  number = {3},
  pages = {127--239},
  year = {2014},
  publisher = {Now Publishers, Inc.}
}

@article{passtyErgodicConvergenceZero1979,
  title = {Ergodic Convergence to a Zero of the Sum of Monotone Operators in {{Hilbert}} Space},
  author = {Passty, Gregory B.},
  year = 1979,
  journal = {Journal of Mathematical Analysis and Applications},
  volume = {72},
  number = {2},
  pages = {383--390},
  publisher = {Elsevier}
}

@article{raydanBarzilaiBorweinChoice1993,
  title = {On the {{Barzilai}} and {{Borwein}} Choice of Steplength for the Gradient Method},
  author = {Raydan, Marcos},
  year = 1993,
  journal = {IMA Journal of Numerical Analysis},
  volume = {13},
  number = {3},
  pages = {321--326},
  publisher = {Oxford University Press}
}

@article{raydanBarzilaiBorweinGradient1997,
  title = {The {{Barzilai}} and {{Borwein}} Gradient Method for the Large Scale Unconstrained Minimization Problem},
  author = {Raydan, Marcos},
  year = 1997,
  journal = {SIAM Journal on Optimization},
  volume = {7},
  number = {1},
  pages = {26--33},
  publisher = {SIAM}
}

@article{rosenbrock1960automatic,
  title = {An Automatic Method for Finding the Greatest or Least Value of a Function},
  author = {Rosenbrock, Howard Harry},
  journal = {The Computer Journal},
  volume = {3},
  number = {3},
  pages = {175--184},
  year = {1960},
  publisher = {Oxford University Press}
}

@article{stellaForwardBackwardQuasiNewton2017,
  title = {Forward--Backward Quasi-{{Newton}} Methods for Nonsmooth Optimization Problems},
  author = {Stella, Lorenzo and Themelis, Andreas and Patrinos, Panagiotis},
  year = 2017,
  journal = {Computational Optimization and Applications},
  volume = {67},
  number = {3},
  pages = {443--487},
  publisher = {Springer}
}

@article{tibshiraniRegressionShrinkageSelection1996,
  title = {Regression Shrinkage and Selection via the Lasso},
  author = {Tibshirani, Robert},
  year = 1996,
  journal = {Journal of the Royal Statistical Society Series B: Statistical Methodology},
  volume = {58},
  number = {1},
  pages = {267--288},
  publisher = {Oxford University Press}
}

@article{wangAdaptiveLearningRate2023,
  title = {Adaptive Learning Rate Optimization Algorithms with Dynamic Bound Based on {{Barzilai-Borwein}} Method},
  author = {Wang, Zhi-Jun and Gao, He-Bei and Wang, Xiang-Hong and Zhao, Shuai-Ye and Li, Hong and Zhang, Xiao-Qin},
  year = 2023,
  month = jul,
  journal = {Information Sciences},
  volume = {634},
  pages = {42--54},
  doi = {10.1016/j.ins.2023.03.050},
  urldate = {2025-12-15},
  langid = {english},
}

@article{wenSurveyNonconvexRegularizationbased2018,
  title = {A Survey on Nonconvex Regularization-Based Sparse and Low-Rank Recovery in Signal Processing, Statistics, and Machine Learning},
  author = {Wen, Fei and Chu, Lei and Liu, Peilin and Qiu, Robert C.},
  year = 2018,
  journal = {IEEE Access},
  volume = {6},
  pages = {69883--69906},
  publisher = {IEEE}
}

@article{wrightSparseReconstructionSeparable2009,
  title = {Sparse Reconstruction by Separable Approximation},
  author = {Wright, Stephen J. and Nowak, Robert D. and Figueiredo, M{\'a}rio A.T.},
  year = 2009,
  journal = {IEEE Transactions on signal processing},
  volume = {57},
  number = {7},
  pages = {2479--2493},
  publisher = {IEEE}
}

@article{yagishitaSimpleLinesearchfreeFirstorder2025a,
  title = {Simple Linesearch-Free First-Order Methods for Nonconvex Optimization},
  author = {Yagishita, Shotaro and Ito, Masaru},
  year = 2025,
  journal = {arXiv preprint arXiv:2509.14670},
  eprint = {2509.14670},
  archiveprefix = {arXiv}
}

@article{yeSimpleAdaptiveProximal2025a,
  title = {A Simple Adaptive Proximal Gradient Method for Nonconvex Optimization},
  author = {Ye, Zilong and Ma, Shiqian and Yang, Junfeng and Zhou, Danqing},
  year = 2025,
  journal = {arXiv preprint arXiv:2510.06079},
  eprint = {2510.06079},
  archiveprefix = {arXiv}
}

@article{zhouAdaBBAdaptiveBarzilaiBorwein2024,
  title = {{{AdaBB}}: {{Adaptive Barzilai-Borwein Method}} for {{Convex Optimization}}},
  shorttitle = {{{AdaBB}}},
  author = {Zhou, Danqing and Ma, Shiqian and Yang, Junfeng},
  year = 2024,
  month = oct,
  journal = {Mathematics of Operations Research},
  pages = {moor.2024.0510},
  doi = {10.1287/moor.2024.0510},
  urldate = {2025-11-24},
  langid = {english}
}
\end{document}